\documentclass[12pt, english, a4paper]{article}
\usepackage[latin9]{inputenc}
\usepackage[T1]{fontenc}
\usepackage{graphicx}
\usepackage{amsthm}
\usepackage{amsfonts,dsfont}
\usepackage{circledsteps}
\usepackage{mathtools}
\allowdisplaybreaks
\usepackage{amssymb}
\usepackage{makecell}
\usepackage{comment}
\usepackage{booktabs}
\usepackage{authblk}
\usepackage{cases}
\usepackage{hyperref}
\usepackage{mathrsfs}
\usepackage{fullpage}
\usepackage{times}
\usepackage{soul}
\usepackage{booktabs}
\usepackage[dvipsnames]{xcolor}
\usepackage{hyperref}
\usepackage{tikz,tikz-cd}
\usetikzlibrary{decorations.text}

\usetikzlibrary{arrows.meta,positioning}
\usepackage{extarrows}
\usepackage{float}
\usepackage{ytableau}
\usepackage{xcolor}
\usepackage{circledsteps}
\usepackage{multirow}
\usepackage{bm}

\usepackage{float}
\usepackage{tikz}
\usepackage{dsfont}
\usetikzlibrary{decorations.text}

\newcommand{\red}[1]{\textcolor{red}{#1}}

\theoremstyle{plain}
\newtheorem{theorem}{Theorem}
\newtheorem{note}{Note}
\newtheorem{corollary}[theorem]{Corollary}
\newtheorem{lemma}[theorem]{Lemma}

\theoremstyle{definition}
\newtheorem{definition}[theorem]{Definition}
\newtheorem{example}{Example}

\newtheorem{question}{Question}

\theoremstyle{remark}

\usepackage{authblk}

\title{\textbf{{Asymmetric  Catalan and Shi Hyperplane Arrangements}}}

\author[1]{{\small Sucharita Biswas}\thanks{\tt{biswas.sucharita56@gmail.com}}}
\affil[1]{{\small Department of Mathematics, Indian Institute of Technology Bombay, Mumbai 400076, India}}

\author[2]{{\small Shubhanshu Prasad}\thanks{\tt{shubhanshup@iisc.ac.in}}}
\affil[2]{{\small Department of Mathematics, Indian Institute of Science, Bengaluru 560012, India}}

\author[3]{{\small Shushma Rani}\thanks{\tt{shushmarani@iitrpr.ac.in}}}
\affil[3]{{\small Department of Mathematics, Indian Institute of Technology Ropar, Punjab 140001, India}}

\date{}
\begin{document}

\maketitle

% \begin{abstract}
%     Motivated by a problem proposed by Theo Douvropoulos during the 2022 Oberwolfach problem session on Enumerative Combinatorics, we introduce and study generalized asymmetric Catalan and Shi arrangements. We define two natural families, namely the first and second kinds, and compute their characteristic polynomials, thereby obtaining explicit formulas for the number of regions. To provide combinatorial interpretations of these region counts, we introduce two new classes of combinatorial objects. For the asymmetric Catalan arrangements, we define $\mathbf{c}$-labelled Dyck paths, which generalize classical Dyck paths, and establish explicit bijections between these paths and the regions of the arrangements. For the asymmetric Shi arrangements, we introduce $\mathbf{c}$-building functions, a natural generalization of parking functions motivated by a dynamic skyscraper construction model, and prove that they are in bijection with the regions of the corresponding arrangements. As special cases, our constructions recover the classical $k$-Catalan and $k$-Shi arrangements. In particular, our results extend Stanley's celebrated bijection between $m$-parking functions and the $m$-Shi arrangement through new combinatorial models and entirely different proof techniques.
% \end{abstract}

\begin{abstract}
In this paper, we study the regions of asymmetric extensions of the Catalan and Shi arrangements. We first determine the characteristic polynomials using the finite field method, thereby obtaining explicit formulas for the number of regions. We then introduce two new classes of combinatorial objects: $\mathbf{c}$-labelled Dyck paths and $\mathbf{c}$-building functions, which generalize classical Dyck paths and parking functions, respectively. We prove that the regions of the asymmetric $\mathbf{c}$-Catalan arrangement are in natural bijection with $\mathbf{c}$-labelled Dyck paths, while those of the asymmetric $\mathbf{c}$-Shi arrangement are in natural bijection with $\mathbf{c}$-building functions. This resolves a problem proposed by Theo Douvropoulos during the 2022 Oberwolfach Workshop on Enumerative Combinatorics. In particular, our results extend Stanley's celebrated correspondence between $k$-parking functions and the regions of the $k$-Shi arrangement by introducing new combinatorial models and developing entirely different proof techniques.
\end{abstract}

\textbf{\small{Keyword:}}{\small{ labelled Dyck path, ordered word, building function, parking function, bijection, enumeration } }{\let\thefootnote\relax}

\section{Introduction}

Hyperplane arrangements form an important area of study connecting combinatorics, geometry, topology, algebra, and representation theory. An affine hyperplane arrangement $\mathcal{A}$ is a finite collection of affine hyperplanes in $\mathbb{R}^n$. The combinatorial structure of an arrangement is described by its \emph{intersection poset} $L(\mathcal{A})$, consisting of all nonempty intersections of hyperplanes ordered by reverse inclusion.
%When every hyperplane passes through the origin, the arrangement is called \emph{central}, and its intersection poset is a geometric lattice. 
An arrangement in $\mathbb{R}^n$ is called
\emph{essential} if its rank is $n$.
A region of $\mathcal{A}$ is a connected component of $\displaystyle \mathbb{R}^n\setminus\bigcup_{H \in \mathcal{A}}H$. A region of a hyperplane arrangement is called \emph{bounded} if it is a bounded subset of the ambient space. More generally, for a non-essential arrangement, a region is bounded if its intersection with the affine subspace spanned by the normal vectors of the hyperplanes is bounded, otherwise we call it unbounded.

One of the main goals in the theory of hyperplane arrangements is to study the regions, that is, the connected components of
$\displaystyle \mathbb{R}^n\setminus\bigcup_{H \in \mathcal{A}}H.$
A fundamental invariant of a hyperplane arrangement is its characteristic polynomial,
$$\chi_{\mathcal{A}}(t)=\sum_{x\in L(\mathcal{A})}\mu(\hat{0},x)t^{\dim(x)},$$
where $\mu$ is the M\"obius function of the intersection poset $L(\mathcal{A})$, and $\hat{0}$ denotes its unique minimal element, namely the ambient space $\mathbb{R}^n$, corresponding to the intersection of the empty family of hyperplanes (Definition 2.52, \cite{orlik1992arrangements}). The characteristic polynomial encodes important combinatorial information about the arrangement.

A celebrated theorem of Zaslavsky relates the characteristic polynomial to the number of regions.

\begin{theorem}[Zaslavsky \cite{zaslavsky1975facing}]\label{Zaslavsky}
For an arrangement $\mathcal{A}$ in $\mathbb{R}^n$,
$$r(\mathcal{A})=(-1)^n\chi_{\mathcal{A}}(-1),$$
where $r(\mathcal{A})$ denotes the number of regions. Moreover,
$$b(\mathcal{A})=(-1)^{\operatorname{rank}(\mathcal{A})}\chi_{\mathcal{A}}(1),$$
where $b(\mathcal{A})$ is the number of bounded regions.
\end{theorem}

While the finite field method \cite{crapo1970foundations, athanasiadis1996characteristic} provides an effective alternate way to compute characteristic polynomials, it does not always make the underlying combinatorial structure apparent. This motivates a more combinatorial approach, that is, to construct explicit bijections between the regions of the arrangement and easily enumerable combinatorial objects, revealing the deep interplay between geometry and combinatorics.

Such connections arise naturally in many important classes of arrangements from combinatorics and Lie theory. For example, the regions of a graphic arrangement \cite{stanley1973acyclic} correspond to the acyclic orientations of a graph, while those of a reflection arrangement \cite{humphreys1990reflection} correspond to the Weyl chambers of the associated Coxeter group. Among these examples, the braid arrangement
$$
\mathcal{A}_{\mathrm{braid}}=\{x_i-x_j=0\mid 1\le i<j\le n\}
$$
is perhaps the most familiar. Its regions are naturally indexed by the permutations of $[n]$, and hence the arrangement has exactly $n!$ regions. An active area of research concerns affine deformations of reflection arrangements obtained by shifting the hyperplanes by integers. There have been many explicit bijections in the literature between the regions of hyperplane arrangements and well-known combinatorial objects such as parking functions, Dyck paths, noncrossing partitions,trees etc.

Some of the most important examples are summarized in Table~\ref{arrangements}.

\begin{table}[H]
\centering
\small
\begin{tabular}{@{}llll@{}}
\toprule
\textbf{Arrangement} & \textbf{Hyperplanes} & \textbf{Number of regions} & \textbf{Combinatorial model} \\
\midrule
Braid \cite{zaslavsky1975facing} & $x_i-x_j=0$ & $n!$ & Permutations \\
Shi \cite{shi1986kazhdan} & $x_i-x_j\in\{0,1\}$ & $(n+1)^{n-1}$ & Parking functions, labelled trees \\

Catalan \cite{postnikov2000deformations} & $x_i-x_j\in\{-1,0,1\}$ & $\displaystyle\frac{n!}{n+1}\binom{2n}{n}$ & Catalan objects \\
Linial \cite{postnikov2000deformations} & $x_i-x_j=1$ &
$\displaystyle\frac1{2^n}\sum_{k=0}^n\binom{n}{k}(k+1)^{n-1}$ &
Alternating trees \\
\bottomrule
\end{tabular}
\caption{Some classical hyperplane arrangements and their combinatorial interpretations.}
\label{arrangements}
\end{table}

Among these, the Shi arrangement, introduced by Shi, has $(1+n)^{n-1}$ regions. Pak and Stanley showed \cite{stanley1996hyperplane} that these regions are in bijection with parking functions and labelled trees. In fact, they proved it for a much more general framework. For a positive integer $k$, they define the generalized Shi arrangement \cite{stanley1998hyperplane} as follow:
\begin{equation}\label{k-shi}
    \mathcal{A}_{\mathrm{Shi}}^k=\{x_i-x_j=s\mid s\in [-k+1, k]\text{ and }1\le i<j\le n  \}.
\end{equation}
They proved that, 
the regions of $\mathcal{A}_{\mathrm{Shi}}^k$ are in bijection with the $k$-Parking function. Consequently, the number of regions is the number of $k$-parking function, $(1+kn)^{n-1}.$

For the Catalan arrangement, each Weyl chamber contributes the same number of regions, given by the Catalan number $C_n$. Thus, the total number of regions is $n!\,C_n=\frac{n!}{n+1}\binom{2n}{n}$. More generally, for a positive integer $k$, the generalized Catalan arrangement is defined by
\begin{equation}\label{k-catalan}
    \mathcal{A}_{\mathrm{Cat}}^k
    =
    \{x_i-x_j=s\mid s\in[-k,k]\text{ and }1\le i<j\le n\}.
\end{equation}
The regions of these arrangements are known to correspond to labelled $(k+1)$-ary trees with $n$ nodes \cite[Section~8.1]{bernardi2018deformations}. These bijective interpretations have since been extended in several directions. For example, Fu, Wang, and Zhu~\cite{fu2021bijections} introduced a cubic matrix associated with each region and used it to construct bijections between regions of the $k$-Shi and $k$-Catalan arrangements and $O$-rooted labelled $k$-trees and pairs of permutations and $k$-Dyck paths, respectively. Duarte and Guedes de Oliveira~\cite{duarte2021pak}, using the Pak-Stanley labeling, characterized the labels of the regions of the $k$-Catalan arrangement and established bijections with $k$-Catalan functions. More broadly, various generalizations and deformations of Weyl arrangements have led to bijections with non-nesting partitions, directed multigraph parking functions, and functional graphs \cite{deshpande2025sketches, bernardi2026bijectivity, bernardi2025deformations}.

Researchers have extensively studied generalizations of these hyperplane arrangements, often finding beautiful bijections between their geometric regions and various combinatorial objects. In recent remarkable work, Bernardi \cite{bernardi2018deformations}  explored deformations of the braid arrangement. For a finite set of integers $S$, he looked at real hyperplane arrangements made up of a finite number of hyperplanes defined as:
\begin{equation*}
    H_{i,j,s} = \{(x_1, \dots, x_n) \in \mathbb{R}^n \mid x_i - x_j = s\},
\end{equation*}
where $1 \le i < j \le n$ and $s \in S$. We call this the $S$-braid arrangement, denoted as $\mathcal{A}_S(n)$. This framework neatly captures many classical examples. For instance, the braid, Catalan, Shi, semiorder, and Linial arrangements correspond to setting $S$ equal to $\{0\}$, $\{-1, 0, 1\}$, $\{0, 1\}$, $\{-1, 1\}$, and $\{1\}$, respectively.

% Building on this, Deshpande and Menon \cite{deshpande2025sketches}  focus on the problem of counting the regions of reflection arrangements and their deformations. Inspired by Bernardi's approach, they show that we can use the concepts of \emph{moves} and \emph{sketches} to create a clear, uniform bijection. Specifically, this bijection directly maps the regions of the Catalan deformation of a reflection arrangement to certain non-nesting partitions.

% Recently, Bernardi and Goregaokar \cite{bernardi2026bijectivity} investigated the generalized Pak-Stanley labeling \cite{stanley1998hyperplane}  for deformations of the braid arrangement. They introduced a specific family of transitive deformations called $(\mathbf{m}, \epsilon)$-arrangements, which generalize the classical $m$-Shi and $m$-Catalan arrangements, and proved that the generalized Pak-Stanley labeling acts as a strict bijection from these regions onto a set of directed multigraph parking functions. 

% Bernardi and Douvropoulos \cite{bernardi2025deformations} studied free Shi-like deformations of restrictions of Weyl arrangements by defining multiplicities through root-theoretic data. In type A, they enumerated the regions of a particular deformation family via a bijection with rooted trees, building on a far-reaching generalization of Joyal's correspondence between labelled trees and functional graphs.  

%the following para is replacement of previous three paras

While previous works have elegantly addressed uniform deformations using algebraic matrices and center-based algorithms, our work expands this landscape by studying asymmetric Catalan and Shi arrangements. More precisely, we consider asymmetric extensions in which the defining hyperplanes are governed by a prescribed tuple $\mathbf{c}$ of positive integers. We first determine the characteristic polynomials of broad classes of these arrangements using the finite field method, obtaining explicit formulas for their numbers of regions. We then develop new combinatorial models for these regions by introducing $\mathbf{c}$-labelled Dyck paths and $\mathbf{c}$-building functions, which generalize classical Dyck paths and parking functions, respectively. Our first main result establishes a bijection between the regions of the asymmetric $\mathbf{c}$-Catalan arrangement and $\mathbf{c}$-labelled Dyck paths.

\begin{theorem}\label{main-bijection-cat1}
There is a bijection between the regions of the asymmetric $\mathbf{c}$-Catalan arrangement and the set of $\mathbf{c}$-labelled Dyck paths.
\end{theorem}

\noindent Similarly, we obtain a bijection between the regions of the asymmetric $\mathbf{c}$-Shi arrangement and $\mathbf{c}$-building functions.

\begin{theorem}\label{main_shi0}
There is a bijection between the regions of the asymmetric $\mathbf{c}$-Shi arrangement and the set of $\mathbf{c}$-building functions.
\end{theorem}

\noindent Together, these results extend the classical correspondence between regions of the Shi arrangement and parking functions to a broader asymmetric setting, while providing new combinatorial objects and bijective constructions.

Our work also resolves a problem posed by Theo Douvropoulos in the \emph{Oberwolfach Problem Session: Enumerative Combinatorics 2022} \cite{williams2022oberwolfach}, Problem 5.2. Given $n$ positive integers $\mathbf{m}=(m_1,\ldots,m_n)$, Douvropoulos considered the arrangements
\begin{equation}\label{theo_cat}
\mathcal{A}^{\mathbf{m}}_{\mathrm{Cat}}=\{x_i-x_j\in\{-m_i-m_j,\ldots,m_i+m_j\}\}
\end{equation}
and
\begin{equation}\label{theo_shi}
\mathcal{A}^{\mathbf{m}}_{\mathrm{Shi}}=\{x_i-x_j\in\{-m_i-m_j+1,\dots,m_i+m_j\}\}.
\end{equation}
He asked for a combinatorial interpretation of the regions of these arrangements. This question motivates our broader framework of asymmetric $\mathbf{c}$-Catalan and $\mathbf{c}$-Shi arrangements, introduced in Sections \ref{sec:ACA} and \ref{sec:ASA}. Our results provide combinatorial interpretations for these arrangements and, in particular, resolve the problem proposed by Douvropoulos; see Notes \ref{theo_cat_sol} and \ref{theo_shi_sol}.

The paper is organized as follows. In Section \ref{sec: generalized cat-shi}, we introduce the asymmetric $\mathbf{c}$-Catalan and $\mathbf{c}$-Shi arrangements and compute their characteristic polynomials. These arrangements generalize the $k$-Catalan and $k$-Shi arrangements, respectively. The main goal of this paper is to identify combinatorial objects that enumerate the regions of these arrangements and to establish explicit bijections between the regions and these objects.

In Section \ref{sec: c dyck}, we introduce $\mathbf{c}$-labelled Dyck paths, a natural generalization of classical Dyck paths in which the up-steps have variable heights $c_i$, with their sequence recorded by a permutation. Alongside this geometric model, we define the generalized Fuss--Catalan number (see Definition \ref{generalized Fuss-Catalan}), which extends the classical Fuss--Catalan sequence to heterogeneous parameters. Using the cycle lemma, we determine the exact number of $\mathbf{c}$-labelled Dyck paths (see Theorem \ref{m dyck}). We then relate these paths to the regions of the asymmetric $\mathbf{c}$-Catalan arrangement through an intermediate class of Catalan $\mathbf{c}$-ordered words. More precisely, we first construct a bijection between the regions of the asymmetric $\mathbf{c}$-Catalan arrangement and Catalan $\mathbf{c}$-ordered words. We next establish a bijection between Catalan $\mathbf{c}$-ordered words and $\mathbf{c}$-labelled Dyck paths. The composition of these two bijections yields the first main theorem.

In Section \ref{sec: c-building}, we introduce a skyscraper construction that leads to the notion of $\mathbf{c}$-building functions, a natural generalization of classical parking functions (see Note \ref{gen of parking}). This construction is motivated by a dynamic skyscraper model in which firms add $c_i$ floors once prescribed height thresholds are reached (see Section \ref{c-building}). Using the matrix-tree theorem, we determine the number of $\mathbf{c}$-building functions (see Theorem \ref{count_building-func}). We then study the regions of the asymmetric $\mathbf{c}$-Shi arrangement in Section \ref{sec:ASA} and construct the corresponding bijection with $\mathbf{c}$-building functions. As in the Catalan case, the construction proceeds through an intermediate class of Shi $\mathbf{c}$-ordered words: we first establish a bijection between the regions of the asymmetric $\mathbf{c}$-Shi arrangement and Shi $\mathbf{c}$-ordered words, and then between Shi $\mathbf{c}$-ordered words and $\mathbf{c}$-building functions. This gives the second main theorem. In particular, our results extend Stanley's bijection \cite{stanley1998hyperplane} between $k$-parking functions and the $k$-Shi arrangement, although our construction and proof are entirely different.

\noindent
\textbf{Notation:}
Throughout, we use $S_n$ to denote the symmetric group of all bijections on
$[n]=\{1,2,\ldots,n\}$. We write permutations in one-line notation. For
integers $m<n$, we use
\[
[m,n]=\{m,m+1,\ldots,n\}
\]
to denote the set of integers from $m$ to $n$. For a set $I$, let
${\mathds{1}}_{I}$ denote its indicator function, defined by
\[
{\mathds{1}}_{I}(i)=
\begin{cases}
1, & \text{if } i\in I,\\
0, & \text{if } i\notin I.
\end{cases}
\]

\section{Characteristic Polynomials}\label{sec: generalized cat-shi}
In this section, we determine the characteristic polynomials for the asymmetric $\mathbf{c}$-Catalan and $\mathbf{c}$-Shi arrangements. Because this polynomial encodes the core topology of a hyperplane arrangement, calculating it allows us to apply Zaslavsky's theorem and directly count the exact number of total and bounded regions. By evaluating these polynomials explicitly, we derive closed product formulas for the region counts across all families of asymmetric arrangements. Beyond simply giving us the enumeration, these formulas lay the groundwork for our investigation and actively motivate the bijective combinatorial models $\mathbf{c}$-labelled Dyck paths and $\mathbf{c}$-building functions that we introduce later in the paper.

\subsection{Characteristic Polynomial of Asymmetric $\mathbf{c}$-Catalan Arrangement}

The classical Catalan arrangement and its $k$-extended counterparts are central to algebraic combinatorics, widely studied for their connections to root systems, Coxeter groups, and Dyck paths. In these classical settings, the bounds on the hyperplane equations $x_i - x_j = k$ are uniform across all pairs. In this section, we move beyond this uniform setting by assigning each coordinate an individual integer parameter from a vector $\mathbf{c} = (c_1, c_2, \dots, c_n) \in \mathbb{Z}_{>0}^n$. By incorporating the parity of each entry, we introduce the \emph{asymmetric $\mathbf{c}$-Catalan arrangement}, which captures asymmetric boundary conditions while generalizing symmetric Catalan deformations. We then compute its characteristic polynomial to obtain an explicit product formula for its total number of regions.

\begin{definition}\label{catalan-kind}
Let $\mathbf{c}=(c_1,c_2,\ldots,c_n)\in\mathbb{Z}_{>0}^n$ be a tuple
of positive integers. For each $i\in[n]$, define  $\mathbf{m}=(m_1,m_2,\ldots , m_n)$ such that
$$
m_i=\begin{cases}
\dfrac{c_i+1}{2}, & \text{if $c_i$ is odd},\\[6pt]
\dfrac{c_i}{2}, & \text{if $c_i$ is even},
\end{cases}
$$
and let $I=\{i\in[n]\mid c_i\text{ is odd}\}$. The \emph{asymmetric $\mathbf{c}$-Catalan arrangement} in $\mathbb{R}^n$ is the hyperplane arrangement
$$
\mathcal{A}_{\mathrm{Cat}}^{\mathbf{c}} = \left\{x_i-x_j=k \;\middle|\; k\in \left[ -m_i-m_j+\mathds{1}_{I}(i),\, m_i+m_j-\mathds{1}_{I}(j)\right]\right\},
$$
where $\mathds{1}_{I}$ denotes the indicator function of $I$.
Equivalently,
$$
\mathcal{A}_{\mathrm{Cat}}^{\mathbf{c}} =\left\{x_i-x_j=k\;\middle|\;
\begin{array}{ll}
i\in I,\ j\notin I, & k\in[-m_i-m_j+1,\,m_i+m_j],\\[4pt]

% \red{i\notin I,\ j\in I,} & k\in[-m_i-m_j,\,m_i+m_j+1],\\[4pt]

i,j\in I, & k\in[-m_i-m_j+1,\,m_i+m_j-1],\\[4pt]
i,j\notin I, & k\in[-m_i-m_j,\,m_i+m_j].
\end{array}
\right\}.
$$
\end{definition}
\noindent Example of the asymmetric $\mathbf{c}$-Catalan arrangement for $\mathbf{c} =(4,2,1)$ is shown in Figure \ref{fig:catalan_regions}.

\begin{theorem}\label{char-catalan}
Let $\mathbf{c}=(c_1,c_2,\dots,c_n)$ be a tuple of positive integers, and let $\displaystyle C=\sum_{i=1}^n c_i$.
Then the characteristic polynomial of the asymmetric $\mathbf{c}$-Catalan arrangement is given by,
$$
\chi(\mathcal{A}^\mathbf{c}_{\mathrm{Cat}},q)=
q\displaystyle\prod_{i=1}^{n-1}(q-C-i)$$
Moreover, the number of regions is
$$
r(\mathcal{A}_{\mathrm{Cat}}^{\mathbf{c}})= \displaystyle\frac{n!}{C+1}\binom{C+n}{n}$$
\end{theorem}

\begin{proof}
We use the finite field method to compute the characteristic polynomial. Therefore, by selecting a prime power $q$ that is sufficiently large, we have:
$$\begin{aligned}
\chi(\mathcal{A}_{\mathrm{Cat}}^{\mathbf{c}}, q) &= \left| \mathbb{F}_q^n \setminus \mathcal{A}_{\mathrm{Cat}}^{\mathbf{c}} \right| \\
=| \{ (\alpha_1, \alpha_2,& \dots, \alpha_n) \in \mathbb{F}_q^n \mid \alpha_i - \alpha_j \notin [-m_{i}-m_{j}+\mathds{1}_{I}(i),m_i+m_j-\mathds{1}_{I}(j)] \pmod q \}| \\
\text{i.e., the } &\text{cyclic distance between } \alpha_i \text{ and } \alpha_j \text{ is strictly greater than } m_i+m_j-\mathds{1}_{I}(j).
\end{aligned}$$
Equivalently, for $i,j\in [n]$,
$$\alpha_i-\alpha_j>m_{i}+m_{j}-\mathds{1}_I(j),$$

Let us represent the elements of the finite field $\mathbb{F}_q=\{\bar{j}: 0 \leq j<q\}$ as $q$ equally spaced points on a circle. Choose a permutation $\sigma= (\sigma(1), \sigma(2), \dots, \sigma(n)) \in S_n$, which specifies the cyclic order in which we place $(\alpha_1, \alpha_2, \dots, \alpha_n)$ on the circle:
$\alpha_{\sigma(1)}, \alpha_{\sigma(2)}, \dots, \alpha_{\sigma(n)}$. Since the circle is invariant under rotation, we may assume without loss of generality that $\alpha_1$ is placed first, i.e., $\sigma(1)=1$. 
We then recursively place $\alpha_{\sigma(1)},\alpha_{\sigma(2)},\ldots,\alpha_{\sigma(n)}$ on the circle according to the following rules.

\begin{itemize}
    \item Place $\alpha_{\sigma(1)}=\alpha_1$ at any of the $q$ possible positions on the circle.

    \item Suppose that $\alpha_{\sigma(1)},\ldots,\alpha_{\sigma(r-1)}$ have already been placed. 
    To ensure that $\alpha_i-\alpha_j>m_{i}+m_{j}-\mathds{1}_I(j)$ for all $i,j\in [n]$, consecutive elements must be separated by a prescribed minimum number of blank positions. 
    Thus, after $\alpha_{\sigma(r-1)}$, leave
    $$m_{\sigma(r-1)}+m_{\sigma(r)}-\mathds{1}_I(\sigma{(r-1)})+a_{r-1}$$
    blank positions before placing $\alpha_{\sigma(r)}$, where $a_{r-1}\geq 0$ records any additional blank positions beyond the required minimum.

    \item Continue this procedure for $r\in [2,n]$, placing each $\alpha_{\sigma(r)}$ after the required number of blank positions following $\alpha_{\sigma(r-1)}$.

    \item Finally, after $\alpha_{\sigma(n)}$, there must be
    $$m_{\sigma(n)}+m_{\sigma(1)}-\mathds{1}_I(\sigma(n))+a_n$$
    blank positions before returning to $\alpha_{\sigma(1)}=\alpha_1$, where $a_n\geq 0$ accounts for any additional blank positions in the final gap.
\end{itemize}

Thus, once the cyclic order $\sigma$ is fixed, the placement is determined by the initial position of $\alpha_1$ together with the nonnegative integers $a_1,\ldots,a_n$, which record the additional blank positions between consecutive elements.
\medskip

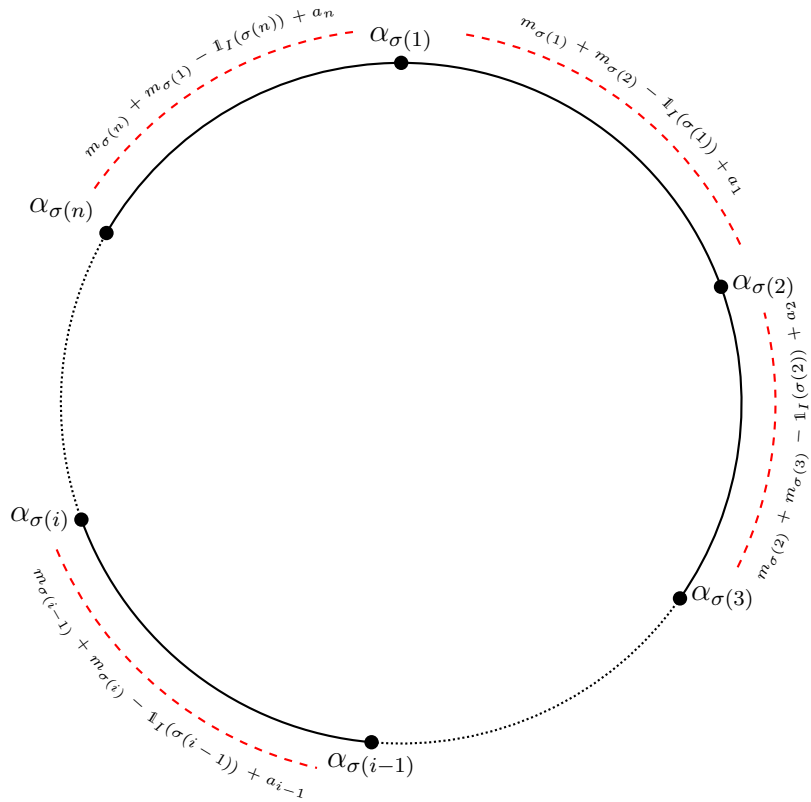
\begin{figure}[H]
    \centering
\begin{tikzpicture}[scale=1.5]

% Radius parameters
\def\r{3}
\def\rr{3.3}

% Points on the circle
\coordinate (xn)  at (150:\r);
\coordinate (x1)  at (90:\r);
\coordinate (x2)  at (20:\r);
\coordinate (x3)  at (-35:\r);
\coordinate (xi)  at (-95:\r);
\coordinate (xip) at (-160:\r);

% Main circle arcs
\draw[thick] (xn)  arc (150:90:\r);
\draw[thick] (x1)  arc (90:20:\r);
\draw[thick] (x2)  arc (20:-35:\r);
\draw[densely dotted,thick] (x3) arc (-35:-95:\r);
\draw[thick] (xi)  arc (-95:-160:\r);
\draw[densely dotted,thick] (xip) arc (-160:-210:\r);

% Vertices
\foreach \p in {xn,x1,x2,x3,xi,xip}
    \fill (\p) circle (1.8pt);

% Vertex Labels
\node[above left] at (xn) {$\alpha_{\sigma(n)}$};
\node[above] at (x1) {$\alpha_{\sigma(1)}$};
\node[right] at (x2) {$\alpha_{\sigma(2)}$};
\node[right] at (x3) {$\alpha_{\sigma(3)}$};
\node[below] at (xi) {$\alpha_{\sigma(i-1)}$};
\node[left] at (xip) {$\alpha_{\sigma(i)}$};

% Red dashed arcs with curved expressions
% Notice how every math component is now wrapped in {}

% Between x_n and x_1 
\draw[red,dashed,thick,postaction={decorate,decoration={text along path,text align=center,raise=1.5ex,
    text={|\tiny| {$m_{\sigma(n)}$} {$+$} {$m_{\sigma(1)}$} {$-$} {$\mathds{1}_{I}(\sigma(n))$} {$+$} {$a_n$}}}}] 
    (145:\rr) arc (145:97:\rr);

% Between x_1 and x_2 
\draw[red,dashed,thick,postaction={decorate,decoration={text along path,text align=center,raise=1.5ex,
    text={|\tiny| {$m_{\sigma(1)}$} {$+$} {$m_{\sigma(2)}$} {$-$} {$\mathds{1}_{I}(\sigma(1))$} {$+$} {$a_1$}}}}] 
    (80:\rr) arc (80:25:\rr);

% Between x_2 and x_3 (Counter-Clockwise)
\draw[red,dashed,thick,postaction={decorate,decoration={text along path,text align=center,raise=-2ex,
    text={|\tiny| {$m_{\sigma(2)}$} {$+$} {$m_{\sigma(3)}$} {$-$} {$\mathds{1}_{I}(\sigma(2))$} {$+$} {$a_2$}}}}] 
    (-26:\rr) arc (-26:14:\rr);

% Between x_i and x_{i+1} (Counter-Clockwise)
\draw[red,dashed,thick,postaction={decorate,decoration={text along path,text align=center,raise=-2ex,
    text={|\tiny| {$m_{\sigma(i-1)}$} {$+$} {$m_{\sigma(i)}$} {$-$} {$\mathds{1}_{I}(\sigma(i-1))$} {$+$} {$a_{i-1}$}}}}] 
    (-157:\rr) arc (-157:-103:\rr);

\end{tikzpicture}
\medskip
\caption{Cyclic placement of points}
\label{fig:cyclic_points}
\end{figure}

%\vspace{0.3cm}
\noindent
The total number of available blank spaces when placing $n$ elements on a circle of $q$ positions is exactly $q - n$. Summing the gaps between all consecutive elements gives:
$$
(m_{\sigma(1)}+m_{\sigma(n)})+\sum_{j=1}^{n-1} \left(m_{\sigma(j)}+m_{\sigma(j+1)}\right) + \sum_{i=1}^n(a_{i}-\mathds{1}_I(\sigma(i)))  = q - n
$$
This simplifies to,
$$
\left(2\sum_{i=1}^nm_i-|I|\right) + \sum_{j=1}^n a_j=C+ \sum_{j=1}^n a_j= q - n 
$$
Therefore , 
$$  
\sum_{j=1}^n a_j = q-n-C.
$$
Hence, finding a valid spacing configuration is equivalent to finding the number of non-negative integer solutions $(a_1, \dots, a_n)$ to the equation above. By the standard formula for weak compositions, the number of such solutions is:
$$ 
\binom{(q - C - n) + n - 1}{n - 1} = \binom{q - C - 1}{n - 1}. 
$$

Thus, for a specific cyclic ordering (permutation $\sigma$), the total number of valid ways to place $(\alpha_1,  \dots, \alpha_n)$ on the circle is:
$$ 
q \binom{q - C - 1}{n - 1}. 
$$
To find the total number of valid choices for the coordinates $(\alpha_1, \dots, \alpha_n)$, we must sum over all valid permutations where $\sigma(1)=1$. There are $(n - 1)!$ such cyclic permutations. 
% $$ 
% \sum\limits_{\{\sigma \in S_n\mid \sigma(1)=1\}} q \binom{q - C - 1}{n - 1} 
% $$
Therefore,
\begin{align*}
\chi(\mathcal{A}_{\mathrm{Cat}}^{\mathbf{c}}, q) &=  \sum\limits_{\{\sigma \in S_n\mid \sigma(1)=1\}} q \binom{q - C - 1}{n - 1}  \\
&=  \sum\limits_{\{\sigma \in S_n\mid \sigma(1)=1\}} q \frac{(q - C - 1)(q - C - 2) \dots (q - C - (n-1))}{(n-1)!}  \\
&=  \frac{q}{(n-1)!}\sum\limits_{\{\sigma \in S_n\mid \sigma(1)=1\}}  \left(\prod_{i=1}^{n-1} (q - C- i)\right)\\
&=  q \prod_{i=1}^{n-1} (q - C- i).
\end{align*}
By Theorem \ref{Zaslavsky}, the number of regions in asymmetric $\mathbf{c}$-Catalan arrangement is 
\begin{align*}
r(\mathcal{A}_{\mathrm{Cat}}^{\mathbf{c}}) &=(-1)^n\chi(\mathcal{A}_{\mathrm{Cat}}^{\mathbf{c}},-1)\\
&=(-1)^n(-1)\prod_{i=1}^{n-1}(-1-C-i)\\
&=\prod_{i=1}^{n-1}(C+i+1)\\
&=(C+2)(C+3)\cdots(C+n)\\
&=\frac{n!}{C+1}\binom{C+n}{n}.
\end{align*}
This completes the proof.
    
\end{proof}

\subsection{ Characteristic Polynomial of Asymmetric $\mathbf{c}$-Shi Arrangement}

Much like our approach to the Catalan case, we now extend the Shi arrangement to an asymmetric setting. While the standard $k$-Shi arrangement shifts lower bounds by $+1$ uniformly, our setup varies these shifts based on the parity vector $\mathbf{c} \in \mathbb{Z}_{>0}^n$. Below, we define the \emph{asymmetric $\mathbf{c}$-Shi arrangement} and derive its characteristic polynomial, leading directly to a closed product formula for its region count.

\begin{definition}\label{shi-kind}
Let $\mathbf{c}=(c_1,c_2,\ldots,c_n)\in\mathbb{Z}_{>0}^n$
be a tuple of positive integers. For each $i\in[n]$, define $\mathbf{m}=(m_1,m_2,\ldots , m_n)$ such that
$$
m_i=
\begin{cases}
\dfrac{c_i+1}{2}, & \text{if } c_i \text{ is odd},\\[6pt]
\dfrac{c_i}{2}, & \text{if } c_i \text{ is even}.
\end{cases}
$$
and let $I=\{i\in[n]\mid c_i\text{ is odd}\}$. The \emph{asymmetric $\mathbf{c}$-Shi arrangement} in
$\mathbb{R}^n$ is the hyperplane arrangement
$$ 
\mathcal{A}_{\mathrm{Shi}}^{\mathbf{c}}=\{ x_i-x_j \in [-m_{i}-m_{j}+1+\mathds{1}_{I}(i),m_i+m_j-\mathds{1}_{I}(j)]\}
$$ 
where $i<j$ and $\mathds{1}_{I}$ denotes the indicator function of $I$. Equivalently,
$$
\mathcal{A}_{\mathrm{Shi}}^{\mathbf{c}}=\left\{
x_i-x_j=k
\;\middle|\;
\begin{array}{ll}
i\in I,\ j\notin I,
&
k\in[-m_i-m_j+2,\,m_i+m_j],\\[4pt]
i\notin I,j\in I,
&
k\in[-m_i-m_j+1,\,m_i+m_j-1],\\[4pt]
i,j\notin I,
&
k\in[-m_i-m_j+1,\,m_i+m_j]\\[4pt]
i,j\in I,
&
k\in[-m_i-m_j+2,\,m_i+m_j-1].
\end{array}
\right\}
$$
where $i<j$.
\end{definition}

\noindent Example of the asymmetric $\mathbf{c}$-Shi arrangement for $\mathbf{c} =(4,2,1)$ is shown in Figure \ref{fig:shi_regions}.

\begin{theorem}
The characteristic polynomial of the asymmetric $\mathbf{c}$-Shi arrangement is given by 
$$\chi(\mathcal{A}^\mathbf{c}_{\mathrm{Shi}}, q) = q(q-C)^{n-1}.$$
Moreover, the number of regions is
% Moreover the number of region is 
$$r(\mathcal{A}_{\mathrm{Shi}}^{\mathbf{c}})=(1+C)^{n-1}.$$
\end{theorem}

\begin{proof}
Let $q$ be a sufficiently large prime power.
%and let $\mathbb{F}_q$ denote the finite field with $q$ elements.
By finite field method, we get  the characteristic polynomial 
$$\chi(\mathcal{A}_{\mathrm{Shi}}^{\mathbf c},q)
=
\left|
\mathbb{F}_q^n\setminus
\mathcal{A}_{\mathrm{Shi}}^{\mathbf c}
\right|.$$
Hence,
\[
\chi(\mathcal{A}_{\mathrm{Shi}}^{\mathbf c},q)
=
\left|\left\{
(\alpha_1,\ldots,\alpha_n)\in\mathbb{F}_q^n \;\middle|\;
\begin{gathered}{\alpha_i-\alpha_j
\notin
[-m_{i}-m_{j}+1+\mathds{1}_I(i),\,
m_{i}+m_{j}-\mathds{1}_I(j)]}\\{
\text{ for all }1\le i<j\le n} \end{gathered}
\right\}\right|.
\]
% Equivalently,
% $$
% \begin{cases}
% \alpha_i-\alpha_j>m_{i}+m_{j}-\mathds{1}_I(j),& i<j,\\[1mm]
% \alpha_j-\alpha_i>m_{i}+m_{j}-\mathds{1}_I(i)-1,& i>j.
% \end{cases}
% $$
Equivalently, when $i<j$, either
$$\alpha_i-\alpha_j>m_{i}+m_{j}-\mathds{1}_I(j),$$
or
$$
\alpha_j-\alpha_i>m_{i}+m_{j}-\mathds{1}_I(i)-1.
$$
Let us represent the elements of the finite field $\mathbb{F}_q=\{\bar{j}: 0 \leq j<q\}$ as points distributed evenly around a circle. 
Choose a permutation $\sigma=(\sigma(1),\sigma(2),\ldots,\sigma(n))\in S_n,$ to specify the cyclic order in which the elements
$\alpha_1,\ldots,\alpha_n$ are placed on the circle. Since rotating the circle does not change the configuration, we may assume without loss of generality that $\sigma(1)=1$. There are $q$ choices for the position of $\alpha_{\sigma(1)}$.
For each consecutive pair
$\alpha_{\sigma(i)},\alpha_{\sigma(i+1)}$, the minimum required number of blank positions depends on whether
$\sigma(i)<\sigma(i+1)$ or
$\sigma(i)>\sigma(i+1)$.
More precisely for $i\geq 2$, between
$\alpha_{\sigma(i)}$ and $\alpha_{\sigma(i+1)}$
there are
% $$
% \begin{cases}
% m_{\sigma(i)}+m_{\sigma(i+1)}-\mathds{1}_I(\sigma(i+1))+a_i,
% &
% \sigma(i)<\sigma(i+1),\\[2mm]
% m_{\sigma(i)}+m_{\sigma(i+1)}-\mathds{1}_I(\sigma(i))+a_i-1,
% &
% \sigma(i)>\sigma(i+1),
% \end{cases}
% $$
$$\begin{cases}
\left(m_{\sigma(i)}+m_{\sigma(i+1)}-\mathds{1}_I(\sigma(i+1))\right)+a_i,
&
\text{if }\sigma(i)<\sigma(i+1),\\[2mm]
\left(m_{\sigma(i)}+m_{\sigma(i+1)}-\mathds{1}_I(\sigma(i+1))-1\right)+a_i,
&
\text{if }\sigma(i)>\sigma(i+1)
\end{cases}$$ blank positions,
where $a_i\ge0$.

% Let $I=\{\sigma(i_1), \sigma(i_2), \ldots, \sigma(i_k)\}$.

% For $1 \leq r \leq i_1-2$, between
% $\alpha_{\sigma(r)}$
% and
% $\alpha_{\sigma(r+1)}$
% there are 
% $$\left\{ \begin{array}{cl}
%     m_{\sigma(r)}+m_{\sigma(r+1)}+a_{r},  & \text{if }\sigma(r) <\sigma(r+1) \\
%      m_{\sigma(r)}+m_{\sigma(r+1)}+a_{r}-1, & \text{if }\sigma(r) >\sigma(r+1)
% \end{array}\right. $$ blank positions.
% For $r=i_1-1$, between
% $\alpha_{\sigma(r)}$
% and
% $\alpha_{\sigma(r+1)}$
% there are $m_{\sigma(r)}+m_{\sigma(r+1)}-1+a_{r}$ blank positions; for $r=i_1$, between
% $\alpha_{\sigma(r)}$
% and
% $\alpha_{\sigma(r+1)}$
% there are  
% $$\left\{ \begin{array}{cl}
%     m_{\sigma(r)}+m_{\sigma(r+1)}+a_{r},  & \text{if }\sigma(r) <\sigma(r+1) \\
%      m_{\sigma(r)}+m_{\sigma(r+1)}+a_{r}-2, & \text{if }\sigma(r) >\sigma(r+1)
% \end{array}\right. $$ blank positions.

% As  the permutation $\sigma$ has $p_{\sigma}= p_\sigma^I+ p_\sigma^{I^c}$ descents,  each descent reduces
% the required number of blank positions by $$\left\lbrace \begin{array}{cl}
%     0, &  \text{if descent occurs at any point in } \{i_1-1, i_2-1, \dots, i_k-1\}\\
%     2, &  \text{if descent occurs at any point in } \{i_1, i_2, \dots, i_k\}\\
%     1, &  \text{otherwise}
% \end{array}\right..$$  

% Therefore, the total number of blank positions equals
% \textcolor{blue}{$$\sum_{j=1}^{n-1} m_{\sigma(j)}+m_{\sigma(j+1)} + m_{\sigma(1)}+m_{\sigma(n)}
% +\sum_{i=1}^n a_i-k-2p_\sigma^I
% -1-p_{\sigma}.$$}
Again, the total number of available blank spaces when placing $n$ elements on a circle of $q$ positions is exactly $(q - n)$. Summing the gaps between all consecutive elements gives:
$$
(m_{\sigma(1)}+m_{\sigma(n)})+\sum_{j=1}^{n-1} \left(m_{\sigma(j)}+m_{\sigma(j+1)}\right) +\sum_{i=1}^n(a_i-\mathds{1}_I(i))-(p_{\sigma}+1)=q-n.
$$
This simplifies to,
$$
\left(2\sum_{i=1}^nm_i-|I|\right) + \sum_{j=1}^n a_j-(p_{\sigma}+1)=C+ \sum_{j=1}^n a_j-(p_{\sigma}+1)= q - n 
$$
Therefore,
$$
\sum_{i=1}^n a_i=q-n+(1+p_{\sigma})-C.
$$
Hence, finding a valid spacing configuration is equivalent to finding the number of non-negative integer solutions $(a_1, \dots, a_n)$ to the equation above. By the standard formula for weak compositions, the number of such solutions is: 
$$
\displaystyle\binom{q-C+p_{\sigma}}{n-1}.
$$ 
Thus, for a specific cyclic ordering (permutation $\sigma$), the total number of valid ways to place $(\alpha_1,  \dots, \alpha_n)$ on the circle is:
$$ 
q \binom{q - C - p_{\sigma}}{n - 1}. 
$$
Summing over all permutations satisfying $\sigma(1)=1$, we obtain
\begin{equation}\label{eq:shi1}
{
\chi(\mathcal{A}_{\mathrm{Shi}}^{\mathbf c},q)
=q\sum_{\{\sigma \in S_n\mid \sigma(1)=1\}}
\binom{q-C+p_{\sigma}}{n-1},
}
\end{equation}
where $p_{\sigma}$ denotes the number of descents of $\sigma$. Since $\sigma(1)=1$, the first position does not contribute to the descent set of $\sigma$, and hence $\sigma$ can have at most $n-2$ descents. Therefore, grouping the permutations in $S_n$ that fix $1$ according to their number of descents, we obtain
$$
\chi(\mathcal{A}_{\mathrm{Shi}}^{\mathbf c},q)=q\sum_{p=0}^{n-2}
\left\langle
\begin{matrix}
n-1\\
p
\end{matrix}
\right\rangle
\binom{q-C+p}{n-1},
$$
where 
$\left\langle
\begin{matrix}
n-1\\
p
\end{matrix}
\right\rangle$
denotes the Eulerian number, counting permutations of $[n-1]$ with exactly $p$ descents. Recall the Worpitzky's identity,
\begin{equation*}
x^m=\sum_{k=0}^{m-1}
\left\langle
\begin{matrix}
m\\
k
\end{matrix}
\right\rangle
\binom{x+k}{m}.
\end{equation*}
Putting $x=q-C$ and $m=n-1$, we obtain
\begin{equation}
(q-C)^{n-1}=\sum_{p=0}^{n-2} \left\langle \begin{matrix}
n-1\\
p
\end{matrix}
\right\rangle
\binom{q-C+p}{n-1}.
\label{eq:shi2}
\end{equation}
Combining \eqref{eq:shi1} and \eqref{eq:shi2} yields
$$
\chi(\mathcal{A}_{\mathrm{Shi}}^{\mathbf c},q)=q(q-C)^{n-1}.
$$
By Theorem \ref{Zaslavsky}, the number of regions in asymmetric $\mathbf{c}$-Shi arrangement is 
\begin{equation*}
r(\mathcal{A}_{\mathrm{Shi}}^{\mathbf c}) =(-1)^n\chi(\mathcal{A}_{\mathrm{Shi}}^{\mathbf c},-1) =(1+C)^{n-1}.
\end{equation*}
This proves the theorem.
\end{proof}

\section{$\mathbf{c}$-labelled Dyck Path}\label{sec: c dyck}

In this section, we introduce a family of $\mathbf{c}$-labelled Dyck paths that provides a combinatorial model for the regions of the asymmetric $\mathbf{c}$-Catalan arrangements. These paths generalize the classical Dyck paths by allowing up steps of varying heights determined by a vector $\mathbf{c}=(c_1,c_2,\ldots,c_n)$, while the order of the up steps is recorded by a permutation. In Section \ref{sec:ACA}, we will establish a bijection between the regions of the asymmetric $\mathbf{c}$-Catalan arrangements and  $\mathbf{c}$-labelled Dyck paths.
Let $\mathbf{c}=(c_1,c_2,\ldots ,c_n)$ be a tuple of $n$ positive integers, and set $\displaystyle 
C=\sum_{i=1}^n c_i .$

\begin{definition}\label{c-dyckpath}
    A \textit{$\mathbf{c}$-labelled Dyck path} is a lattice path of length $(C+n)$ from $(0,0)$ to $(C+n,0)$ consisting of down steps $D=(1,-1)$ and up steps defined as follows: the $i$-th up step is $U_{\pi(i)}=(1,c_{\pi(i)})$,  for some $\pi \in S_n$, for $1\leq i\leq n$.
    
    This path is required to stay weakly above the $x$-axis. Hence, each $\mathbf{c}$-labelled Dyck path contains exactly $n$ up steps and $C$  down steps, giving a total of $(C+n)$ steps. 

    We denote the set of all $\mathbf{c}$-labelled Dyck paths of length $(C+n)$ by  $\mathcal{D}({\mathbf{c}})$.
\end{definition}

\begin{example} Let us consider an example where $n=2$, $\mathbf{c}=(3,2)$, and $\pi=21 \in S_2$. Thus, the first up step is $U_{\pi(1)} = U_2 = (1,c_2) = (1,3)$, and the second up step is $U_{\pi(2)} = U_1 = (1,c_1) = (1,2)$. Consider the $\mathbf{c}$-labelled Dyck path $U_2DU_1DDDD$, shown in the following figure:
%----------%----
 
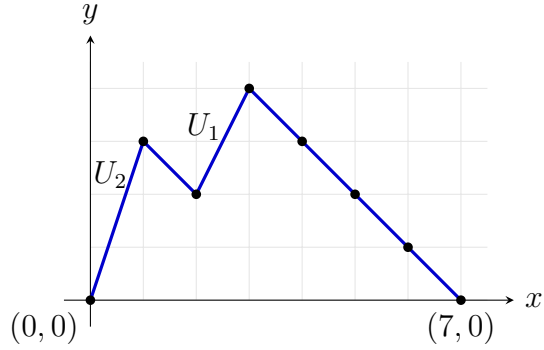
\vspace{3mm}
\begin{center}
\begin{tikzpicture}[scale=0.7, thick]
    % Optional: Add a faint grid to make the discrete steps easier to read
    \draw[thin, gray!20, step=1] (0,0) grid (7.5, 4.5);

    % Axes
    \draw[->, >=stealth, thin] (-0.5,0) -- (8,0) node[right] {$x$};
    \draw[->, >=stealth, thin] (0,-0.5) -- (0,5) node[above] {$y$};

    % Dyck Path
    \draw[blue!80!black, line width=1.2pt] 
        (0,0) -- (1,3) -- (2,2) -- (3,4) -- (4,3) -- (5,2) -- (6,1) -- (7,0);

    % Vertices
    \foreach \p in {(0,0), (1,3), (2,2), (3,4), (4,3), (5,2), (6,1), (7,0)}
        \filldraw[black] \p circle (2pt);

    % Step Labels
    \node[above left] at (0.92, 1.9) {$U_2$};
    \node[above right] at (1.6, 2.8) {$U_1$};

    % Coordinate Labels
    \node[below left] at (0,0) {$(0,0)$};
    \node[below] at (7,0) {$(7,0)$};
\end{tikzpicture}
\end{center}
\end{example}

\begin{note}
If $c_1=c_2=\cdots =c_n=m$, then a $\mathbf{c}$-labelled Dyck path reduces to the classical $m$-Dyck path. Thus, $\mathbf{c}$-labelled Dyck paths form a natural generalization of $k$-Dyck paths.
\end{note}
The following quantity naturally appears in the enumeration of $\mathbf{c}$-labelled Dyck paths.
\begin{definition}[Generalized Fuss-Catalan number]\label{generalized Fuss-Catalan}
Let $\mathbf{c}=(c_1,c_2,\ldots ,c_n)$ be a tuple of $n$ positive integers.
 Suppose that the entries of $\mathbf{c}$ take $l$ distinct values, occurring with multiplicities $k_1,k_2,\ldots ,k_l$, respectively. Hence $k_1+k_2+\cdots +k_l=n.$
The \it{generalized Fuss--Catalan number} associated with $\mathbf{c}$ is defined by
$$
C_n^{\mathbf{c}}
=\frac{n!}{k_1!\,k_2!\cdots k_l!}\,
\frac{1}{(C+1)}\binom{C+n}{n}.
$$
\end{definition}
\begin{note}
    If $c_1=c_2=\cdots =c_n=m$ then $C=mn$ and $l=1$, $k_1=n$, hence 
    $$
    C_n^{\mathbf{c}}=\frac{1}{mn+1}\binom{(m+1)n}{n},
    $$
    which is the $m$-Fuss-Catalan number $C^m_n$. 
\end{note}

\begin{lemma}[Cycle lemma]
Let $a_1,a_2,\ldots, a_n$ be integers with total sum $a_1+a_2+\ldots +a_n=k>0.$ Consider all cycle rotations of the sequence, i.e., 
$$a_i,a_{i+1}, \ldots , a_n,a_1,a_2,\ldots, a_{i-1}.$$
Then exactly $k$ of these cyclic rotations have all partial sums strictly positive.  
\end{lemma}

\begin{corollary}\label{cor:inv_cyc_lemma}
Let $a_1,a_2,\ldots, a_n$ be integers with total sum $a_1+a_2+\ldots +a_n=k>0.$ Consider all inverse cyclic rotations of the sequence. Then exactly $k$ of these cyclic rotations have all inverse partial sums are strictly negative, i.e., $a_n,a_{n-1},\dots,a_{i}$ strictly negative for all $i\in [n]$.  
\end{corollary}
This follows immediately from the Cycle Lemma by reversing the sequence.
% \begin{proof}
% We reduce this result to the standard Cycle Lemma via a reflection argument. Define a dual sequence $b_1, b_2, \ldots, b_n$ by reversing the order of the original terms and inverting their signs:
% $$b_j = -a_{n - j + 1} \quad \text{for } 1 \leq j \leq n.$$

% The total sum of this new sequence is given by $\sum_{j=1}^n b_j = -\sum_{i=1}^n a_i = -k$. Since $k < 0$, the new total sum $-k$ is a strictly positive integer. By the standard Cycle Lemma, exactly $-k$ cyclic rotations of the sequence $b$ have all their partial sums strictly positive.

% Let $b^{(j)}$ denote the cyclic rotation of $b$ starting at index $j$. A direct index translation reveals that the partial sums of $b^{(j)}$ are strictly positive if and only if the partial sums of the cyclic rotation $a^{(n-j+2)}$ (with indices taken modulo $n$) are strictly negative when evaluated from right to left. Because a cyclic sequence has strictly negative partial sums from right to left if and only if its forward partial sums are strictly negative, each valid rotation of $b$ maps bijectively to a unique valid rotation of $a$. Consequently, exactly $-k$ cyclic rotations of the original sequence have exclusively strictly negative partial sums.
% \end{proof}

\begin{theorem}\label{m dyck}
The number of $\mathbf{c}$-labelled Dyck paths is  
$(k_1!\,k_2!\cdots k_l!)C_n^{\mathbf{c}}$.
\end{theorem}

\begin{proof}
Let $\mathcal{L}$ be the number of lattice paths from $(0,0)$ to $(C+n+1,-1)$ where the up steps are $(1,c_{\pi(1)}),\dots, (1,c_{\pi(n)})$ arising in some order and consisting of $C+1$ down steps $(1,-1)$. The number of paths in $\mathcal{L}$ is 
$$
|\mathcal{L}|=n! \binom{C+n+1}{n}=\frac{(C+n+1)!}{(C+1)!}.
$$
Define a relation $\sim$ on $\mathcal{L}$, where two path $L_1\sim L_2$ if $L_2$ is cycle rotation of the sequence of $L_1$. It is easy to observe that the relation $\sim$ is an equivalence relation, and each equivalence class of $\mathcal{L}$ contains exactly $(C+n+1)$ paths.

Assign each of the down path a weigh $1$ and each up step $(1,c_i)$ a weight $-c_{i}$. Therefore, regardless of the chosen permutation $\pi$, the total weight of any such path is strictly constant.
$$(C+1) -\sum_{i=1}^{n}c_{\pi(i)} = (C+1) - C =1.$$

By Corollary \ref{cor:inv_cyc_lemma}, exactly one path in each of the equivalence class defined by $\sim$ on $\mathcal{L}$ will have inverse partial sum strictly positive. We call such a path to be a \textit{good path}. The number of good paths in $\mathcal{L}$ is thus 
$$\frac{n!}{C+n+1} \binom{C+n+1}{n} = \frac{n!}{C+1}\binom{C+n}{n}.$$

Removing the last down step from the good path shifts the partial sums from strictly positive to nonnegative. Therefore, all partial sums of the path obtained is nonpositive, implying that for any $i$, the weight of the all the up steps till the $i$-th step is greater than the weight of all the down steps till the $i$-th step. Therefore this yields a path of length $(C+n)$ with $n$ up steps and \red{$s$} down steps that stays weakly above the $x$-axis. Thus, the resulting path is a valid $\mathbf{c}$-labelled Dyck path corresponding to the permutation $\sigma$. This operation is reversible, establishing a bijection, meaning the total number of geometrically distinct $\textbf{c}$-labelled Dyck paths is exactly the number of good paths. This proves the theorem.
\end{proof}

\section{$\mathbf{c}$-building Function}\label{sec: c-building}

In this section, We introduce a skyscraper model for a heterogeneous family of directed parking functions, which we call $\mathbf{c}$-building functions. For the classical Shi arrangement, the regions are naturally described by parking functions. However, this approach does not extend well to the asymmetric setting. To address this, we introduce a new combinatorial object, called a $\mathbf{c}$-building function, motivated by a dynamic skyscraper construction process. This model provides a simple and intuitive way to describe the regions of asymmetric $\mathbf{c}$-Shi arrangement and forms the basis for the bijections developed in the following sections.

Although the sequence studied in this paper is a natural generalization of the classical $k$-parking function, the usual parking interpretation becomes less intuitive in this setting. To motivate our construction, we introduce an equivalent interpretation in terms of a skyscraper construction project. In Section \ref{sec:ASA}, we will establish a bijection between the regions of the asymmetric $\mathbf{c}$-Shi arrangements and  $\mathbf{c}$-building function.

\subsection{Physical Interpretation: The Skyscraper Construction Model}\label{c-building}

Consider the construction of a skyscraper by $n$ firms indexed by $[n]={1,2,\dots,n}$. The construction proceeds according to the following rules.

\begin{enumerate}
\item Each firm $i\in[n]$ is assigned to construct exactly $c_i$ floors, where $\mathbf{c}=(c_1,c_2,\dots,c_n)$.

\item Firm $i$ can begin construction only after the building has reached a height of at least $a_i$ floors. We call $a_i\in\mathbb{Z}_{\ge0}$ the \emph{initiation threshold} of firm $i$.

\item Initially, the building has height $0$. Whenever the current height is at least the initiation threshold of a firm that has not yet begun construction, that firm may start construction and immediately adds $c_i$ floors to the building. This increase in height may enable additional firms to begin construction, resulting in a chain of successive construction phases.

\end{enumerate}
Thus, instead of viewing the process as cars competing for parking spaces, we interpret it as a sequence of firms whose ability to begin construction depends on the height of the partially completed building. We are interested in those tuples $(a_1,a_2,\dots,a_n)$ for which the resulting chain reaction ensures that every firm is eventually able to begin construction. Equivalently, the construction proceeds without stalling until all $n$ firms have successfully completed their construction. Such a sequence $(a_1,a_2,\dots,a_n)$ is called a \emph{$\mathbf{c}$-building function}. We denote the set of all $\mathbf{c}$-building functions of length $n$ by  $\mathcal{B}({\mathbf{c}})$.

% \begin{definition}
% A sequence of thresholds $a = (a_1, a_2, \dots, a_n)$ is defined as a \textit{$\mathbf{c}$-building function} (or mathematically, a multivariate $\mathbf{c}$-parking function) if the resulting chain reaction guarantees that all $n$ firms successfully mobilize and complete the skyscraper without the project stalling in a deadlock.
% \end{definition}

This constructive framework perfectly mirrors the algebraic requirement that for any immobilized subset of firms $S$, the accumulated floors built by the active firms outside of $S$ must be tall enough to satisfy the threshold $a_i$ of at least one firm $i \in S$.

For formal rigor, we translate the physical movement into a dynamic algorithmic process, and then establish its equivalence to the standard combinatorial subset definition.

\medskip

\noindent\textbf{Dynamic Building Process:} Let $\mathbf{c} = (c_1, c_2, \dots, c_n)$ be the capacity vector (representing floors built per firm), and $a = (a_1, a_2, \dots, a_n) \in \mathbb{Z}_{\ge 0}^n$ be the initiation threshold sequence.

\begin{itemize}
    \item \textbf{Initialization:} 
    Let $M_0 = \emptyset$ be the initial set of mobilized firms. Let $W_0 = [n]$ be the set of waiting firms. Let $H_0 = 0$ be the initial height of the skyscraper in floors.
    
    \item \textbf{Iteration Step $t$ (for $1 \le t \le n$):}
    \begin{enumerate}
        \item Evaluate the current accumulated height of the skyscraper built by the firms that mobilized in previous steps: 
        $H_{t-1} = \sum_{j \in M_{t-1}} c_j$
        
        \item Identify the set of eligible firms $E_t \subseteq W_{t-1}$ whose structural prerequisites are met and are ready to begin construction:
        $E_t = \{ i \in W_{t-1} \mid a_i \le H_{t-1} \}$
        
        \item If $E_t = \emptyset$ while $W_{t-1} \neq \emptyset$, the algorithm terminates in a deadlock (the project stalls).
        
        \item If $E_t \neq \emptyset$, select any firm $v \in E_t$ to mobilize. (Since the building height monotonically increases, this process is abelian; the choice of $v$ does not affect the ultimate success of the project).
        
        \item Update the system state for the next iteration: $M_t = M_{t-1} \cup \{v\}$ and $W_t = W_{t-1} \setminus \{v\}$. 
        %(Note: This recursively defines the mobilized set $M_t$, which will then be evaluated as $M_{t-1}$ in the subsequent step).
    \end{enumerate}
    
    \item \textbf{Termination:} The sequence $a$ is called a $\mathbf{c}$-building function if the algorithm terminates with $M_n = [n]$ and $W_n=\emptyset$ (all firms successfully mobilized and the skyscraper is completed).
\end{itemize}

See the algorithmic execution of Example \ref{valid-c} and \ref{invalid-c} for further clarification. To demonstrate the dynamics of the multivariate $\mathbf{c}$-building function and its reliance on specific firm identities, let us examine a project with $n=3$ firms with a completely heterogeneous capacity vector $\mathbf{c} = (2, 1, 3).$ 

\begin{example}[A Valid $\mathbf{c}$-building Function]\label{valid-c}
   Consider the preference sequence $a = (0, 2, 3)$. We claim this is a valid multivariate $\mathbf{c}$-building function.

% \textbf{Mathematical Verification:}
% By definition, for every non-empty $S \subseteq \{1, 2, 3\}$, there must exist an $i \in S$ such that $a_i \le \sum_{j \notin S} c_j$. We verify a few critical subsets:
% \begin{itemize}
%     \item Let $S = \{3\}$. The firms outside $S$ are $\{1, 2\}$. The condition requires $a_3 \le c_1 + c_2 \implies 3 \le 2 + 1 \implies 3 \le 3$. This holds.
%     \item Let $S = \{2, 3\}$. The firm outside $S$ is $\{1\}$. The condition requires $a_2 \le c_1$ \textit{or} $a_3 \le c_1$. Checking the values, $a_2 = 2 \le 2$. Since at least one firm satisfies the inequality, the subset condition holds.
%     \item Let $S = \{1, 2, 3\}$. The set outside $S$ is $\emptyset$. The condition requires $a_1 \le 0$, $a_2 \le 0$, or $a_3 \le 0$. Since $a_1 = 0 \le 0$, it holds.
% \end{itemize}

% \textbf{Algorithmic Execution:}
Physically, the project initializes with $H_0=0$ built floors and $M_0 = \emptyset$, $W_0=[3]$. Firm $1$ (requiring $a_1=0$) mobilizes, adding $c_1=2$ floors, hence $M_1=\{1\}$ and $W_1=\{2,3\}$. The accumulated height is now $H_1=2$. Firm $2$ (requiring $a_2=2$) mobilizes, adding $c_2=1$ floor, hence $M_2=\{1,2\}$ and $W_2=\{3\}$. The accumulated height is now $H_2=2+1=3$. Firm $3$ (requiring $a_3=3$) mobilizes, hence $M_3=[3]$ and $W_3=\emptyset$. All firms have mobilized without the project stalling. 
\end{example}

\begin{example}[An Invalid Sequence]\label{invalid-c}
Consider the  preference sequence $b = (0, 4, 3)$. We will demonstrate that this sequence results in a project deadlock.

% \textbf{Mathematical Verification:}
% To prove a sequence is invalid, it is sufficient to find a single subset $S$ that violates the condition. Let $S = \{2, 3\}$. The total floors built by firms outside $S$ is simply $c_1 = 2$.
% The condition demands that at least one firm in $S$ must satisfy the inequality:
% \begin{align*}
% b_2 &\le c_1 \implies 4 \le 2 \quad \text{(False)} \nonumber \\
% b_3 &\le c_1 \implies 3 \le 2 \quad \text{(False)}.
% \end{align*}
% Because neither firm in the subset satisfies the threshold, the condition fails. Thus, $b$ is not a $\mathbf{c}$-building function.

% \textbf{Algorithmic Execution:}
%Tracing this physically reveals the exact moment of failure. 
The project initializes with $H_0=0$ built floors and $W_0=[3]$. Firm $1$ (requiring $b_1=0$) mobilizes and builds $c_1 = 2$ floors. The unmobilized queue is now $W_1=\{2, 3\}$, and the accumulated height is $H_1=2$. 
Firm $2$ requires $b_2=4$ floors and Firm $3$ requires $b_3=3$ floors. Because the current height $H_1$ is strictly less than both thresholds, neither firm can proceed. The project deadlocks, confirming the sequence is invalid.
\end{example}

\begin{lemma}[Equivalent Condition]\label{equivalence}
A sequence $a = (a_1, \dots, a_n) \in \mathbb{Z}_{\ge 0}^n$ is a $\mathbf{c}$-building function if and only if for every non-empty subset $S \subseteq [n]$, there exists at least one firm $i \in S$ such that $a_i \le \sum_{j \notin S} c_j$.
\end{lemma}

\begin{proof}
For the forward implication, suppose the subset condition fails for some non-empty $S \subseteq [n]$. This implies that for all $i \in S$, $a_i > \sum_{j \notin S} c_j$. If the algorithm runs and manages to mobilize all firms outside of $S$ (so the mobilized set is $M = [n] \setminus S$), the maximum accumulated height generated is exactly $H = \sum_{j \notin S} c_j$. At this point, the remaining waiting firms are exactly the set $S$. Since $a_i > H$ for every $i \in S$, the eligible set is $E = \emptyset$. The project deadlocks. Thus, the condition is necessary.

For the converse part, assume the subset condition holds for all $S \subseteq [n]$. Suppose, for the sake of contradiction, that the dynamic algorithm deadlocks at step $t$, leaving a non-empty set of unmobilized firms $W_{t-1} = S$. The total height built by the mobilized firms is $H_{t-1} = \sum_{j \notin S} c_j$. Because the system is deadlocked, no firm in $S$ is eligible, meaning $a_i > H_{t-1}$ for all $i \in S$. Consequently, $a_i > \sum_{j \notin S} c_j$ for all $i \in S$. This contradicts our assumption that every subset contains at least one firm satisfying the inequality. Therefore, the algorithm must successfully terminate.
\end{proof}

\begin{note}\label{gen of parking}
The multivariate $\mathbf{c}$-building function generalizes the classical $k$-parking function. Recall that a $k$-parking function is a sequence $(a_1,\ldots,a_n)$ of nonnegative integers whose increasing rearrangement
$a_{(1)}\le \cdots \le a_{(n)}$ satisfies
$a_{(i)}\le k(i-1), \qquad 1\le i\le n.$

If $c_j=k$ for every $j\in[n]$, then the defining condition of a $\mathbf{c}$-building function becomes
$$a_i\le \sum_{j\notin S}c_j=k(n-|S|)$$
for some $i\in S$ and every nonempty subset $S\subseteq[n]$. Choosing $S$ to be the set of indices corresponding to the $n-m+1$ largest entries of $a$, we have $|S|=n-m+1$, so there exists $i\in S$ such that
$a_i\le k(m-1).$
Since the smallest element of $S$ is $a_{(m)}$, it follows that
$a_{(m)}\le k(m-1),$
which is exactly the defining condition of a classical $k$-parking function. Thus, when the capacity vector is constant, the $\mathbf{c}$-building function specializes to the classical $k$-parking function.
\end{note}

\subsection{Enumeration Theorem}
We now enumerate the $\mathbf{c}$-building functions. Before that we need some definitions and theorems for enumeration.

\begin{definition}
    Let $G = (V, E)$ be a sink-rooted digraph, and let the vertices in $V$ be indexed from $1$ to $n$, where the $n$-th node is the sink. The \textit{adjacency matrix} of $G$ is the $n \times n$ matrix $A = [a_{ij}]$, where $a_{ij}$ is the number of edges that start at $i$ and end at $j$ and $\text{deg}^+(i)$ is the number of outward edges from $i$.  The \textit{diagonal matrix} of $G$ is the $n \times n$ diagonal matrix $D = [d_{ij}]$, where
\begin{align*}
d_{ij} = \begin{cases}
    \text{deg}^+(i), & \text{if } i = j \\
    0, & \text{otherwise}
\end{cases}
\end{align*}
The \textit{Laplacian} of $G$ is the matrix $D - A$. The \textit{sink-reduced Laplacian} of $G$ is the $(n - 1) \times (n - 1)$ matrix $L$ formed by removing the nth row and column from the Laplacian of $G$, the row and column corresponding to the sink. To clarify, the sink-reduced Laplacian of $G$ is the matrix $L = [l_{ij}]$, where, for all $1 \le i, j \le n - 1$,
\begin{align*}
l_{ij} = \begin{cases}
    \text{deg}^+(i) - a_{ii}, & \text{if } i = j \\
    -a_{ij}, & \text{otherwise}
\end{cases}
\end{align*}
\end{definition}
An \textit{oriented spanning trees} $T$ of the digraph $G$ is a subgraph $T\subset G$ such that
there exists a unique directed path in $T$ from any vertex $i$ to the root $n$.

\begin{theorem}[Tutte's Matrix-Tree Theorem for Digraphs(\cite{margoliash2010matrix}, Theorem 2.8)]\label{matrix-tree}
    Let $G$ be a sink-rooted digraph. Then the number of oriented spanning trees 
    %(or  reverse arborescence subgraphs) 
    of $G$  is equal to the determinant of the sink-reduced Laplacian $L$ of $G$.
\end{theorem} 
\begin{definition}[$G$-parking function]\label{G-parking}
    For a subset $I$ in $\{1, \dots, n\}$ and a vertex $i \in I$, let
\begin{align*}
d_I(i) = \sum_{j \notin I} a_{ij},
\end{align*}
i.e., $d_I(i)$ is the number of edges from the vertex $i$ to a vertex outside of the subset $I$. Let us say that a sequence $b = (b_1, \dots, b_n)$ of nonnegative integers is a $G$-parking function if, for any nonempty subset $I \subseteq \{1, \dots, n\}$, there exists $i \in I$ such that $b_i < d_I(i)$.
\end{definition}
If $G = K_{n+1}$ is the complete graph on $n+1$ vertices, then $K_{n+1}$-parking functions are the usual parking functions of size $n$.

\begin{theorem}(\cite{Alex-Boris}, Theorem 2.1)\label{no of G parking}
    The number of $G$-parking functions equals the number of oriented spanning trees of the digraph $G$.
\end{theorem}
% by By Tutte's Directed Matrix-Tree Theorem \ref{matrix-tree}, this count is the determinant of the reduced Laplacian matrix $L$.

\begin{theorem}\label{count_building-func}
The number of $\mathbf{c}$-building functions of length $n$ associated with the capacity vector $\mathbf{c} = (c_1, c_2, \dots, c_n)$ is given by
$( 1 + C )^{n-1}$, where $\displaystyle C=\sum_{j=1}^n c_j.$
\end{theorem}

\begin{proof}
We reformulate the problem via directed graphs. Let $a = (a_1, \dots, a_n) \in \mathbb{Z}_{\ge 0}^n$ is a $\mathbf{c}$-building function. By Lemma \ref{equivalence} we have
the sequence condition
$$
a_i < 1 + \sum_{j \notin S} c_j.
$$ This mirrors the survival condition of a $G$-parking function \ref{G-parking} on a directed multigraph with vertex set $V = \{q, 1, 2, \dots, n\}$, where $q$ is a designated sink.

Construct the edges of $G$ as follows:
\begin{itemize}
    \item From every vertex $i \in [n]$, draw exactly $1$ directed edge to the sink $q$.
    \item From every vertex $i \in [n]$, draw $c_j$ directed edges to vertex $j$ (for all $i \neq j$).
\end{itemize}
For any $S \subseteq [n]$, the out-degree of a vertex $i \in S$ to vertices outside of $S$ is 
$$
\text{deg}_{V \setminus S}(i) = 1 + \sum_{j \notin S} c_j.
$$ 
By Theorem \ref{no of G parking}, the number of valid sequences $a$ is equal to the number of oriented spanning trees of the digraph $G$ rooted at $q$.

By Tutte's Directed Matrix-Tree Theorem \ref{matrix-tree}, this count is the determinant of the reduced Laplacian matrix $L$ obtained by deleting the row and column corresponding to $q$. For $i, j \in [n]$, the entries of the $n \times n$ matrix $L$ are:
$$ L_{ij} = -c_j \quad (i \neq j), \quad
L_{ii} = 1 + \sum_{k \neq i} c_k = 1 + C - c_i.$$
 Therefore we can write $L = (1 + C)I - \mathbf{1}\mathbf{c}^T,$
where $I$ is the identity matrix, $\mathbf{1}$ is the all-ones column vector, and $\mathbf{c}^T = (c_1, \dots, c_n)$. 
Applying the Matrix Determinant Lemma we get,
%$\det(A - \mathbf{u}\mathbf{v}^T) = \det(A)(1 - \mathbf{v}^T A^{-1} \mathbf{u})$:
\begin{align*}
\det(L) &= \det((1+C)I) \left( 1 - \mathbf{c}^T \left( \frac{1}{1+C} I \right) \mathbf{1} \right) \nonumber \\
&= (1+C)^n \left( 1 - \frac{1}{1+C} \sum_{j=1}^n c_j \right) \nonumber \\
&= (1+C)^n \left( 1 - \frac{C}{1+C} \right) \nonumber \\
&= (1+C)^n \left( \frac{1}{1+C} \right) = (1+C)^{n-1}.
\end{align*}
This completes the proof.
\end{proof}

\section{Asymmetric $\mathbf{c}$-Catalan Arrangement}\label{sec:ACA}

In this section, we study the asymmetric $\mathbf c$-Catalan arrangement.This naturally extends the previously studied $k$-Catalan arrangements.
Our main goal is to give a bijective description of the regions of this arrangement. To this end, we establish a bijection between its regions and the $\mathbf{c}$-labelled Dyck paths introduced in Section \ref{sec: c-building}.

Given a tuple $\mathbf{c}=(c_1,c_2,\dots,c_n)$, we will show that the regions of the asymmetric $\mathbf{c}$-Catalan arrangement are in bijection with the $\mathbf{c}$-labelled Dyck paths.
Recall that the  asymmetric $\mathbf{c}$-Catalan arrangement $\mathcal{A}_{\mathrm{Cat}}^{\mathbf{c}}$ is the hyperplane arrangement defined in Definition~\ref{catalan-kind}. 

For each $i\in[n]$, define $\mathbf{m}=(m_1,m_2,\ldots , m_n)$ such that
$$
m_i=
\begin{cases}
\dfrac{c_i+1}{2}, & \text{if } c_i \text{ is odd},\\[6pt]
\dfrac{c_i}{2}, & \text{if } c_i \text{ is even}.
\end{cases}
$$ and let $I=\{i\in[n]\mid c_i\text{ is odd}\}$.  The asymmetric $\mathbf{c}$-Catalan arrangement is the hyperplane arrangement
$$ \mathcal{A}_{\mathrm{Cat}}^{\mathbf{c}}=\{ x_i-x_j \in [-m_{i}-m_{j}+\mathds{1}_{I}(i),m_i+m_j-\mathds{1}_{I}(j)]\}
.$$
We denote the set of regions of asymmetric $\mathbf{c}$-Catalan arrangement by $\mathcal{R}_{C}(\mathbf{c})$.
% Given a tuple $\mathbf{c}=(c_1,c_2,\dots,c_n)$, we show that there exists a hyperplane arrangement whose regions are in bijective correspondence with the $\mathbf{c}$-labelled Dyck paths and $\mathbf{c}$-building function. 

% Throughout this section, for each $i\in[n]$, define
% $$
% m_i=
% \begin{cases}
% \dfrac{c_i+1}{2}, & \text{if } c_i \text{ is odd},\\[6pt]
% \dfrac{c_i}{2}, & \text{if } c_i \text{ is even},
% \end{cases}
% $$
% and let 
% $$
% I=\{i\in[n]\mid c_i\text{ is even}\}.
% $$
\begin{example}\label{ex:catalan_regions}
Let us consider an example where $n=3$, $\mathbf{c}=(4,2,1)$, hence $\mathbf{m}=(2,1,1)$ and $I=\{3\}.$  Therefore 
\[
\mathcal{A}_{\mathrm{Cat}}^{\mathbf{c}}=\left\{
x_1-x_2 \in [-3,3],\quad 
x_1-x_3 \in[-3,2],\quad
x_2-x_3 \in[-2,1]\right\}.
\] 
The arrangement is shown in Figure \ref{fig:catalan_regions}. It has a total of $90$ regions. 
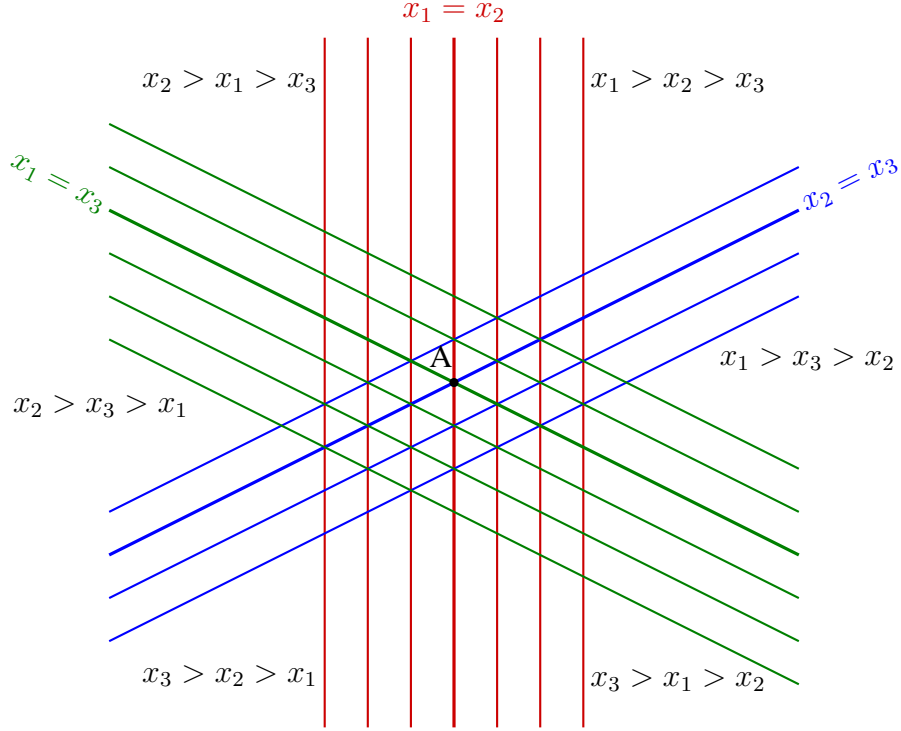
\begin{figure}[htbp]
\centering
\begin{tikzpicture}[xscale=1,yscale=1,scale=0.57]

\colorlet{redline}{red!80!black}
\colorlet{greenline}{green!50!black}

%---------------------------------------------------
% Hyperplanes x1-x2=c
%---------------------------------------------------
\foreach \c in {-3,-2,-1,1,2,3}
{
    \draw[redline,thick] (\c,-8) -- (\c,8);
}

% Central hyperplane x1=x2
\draw[redline,very thick] (0,-8) -- (0,8);
\node[text=redline, above] at (0,8.1) {$x_1=x_2$};

%---------------------------------------------------
% Hyperplanes x2-x3=d
%---------------------------------------------------
\foreach \d in {-1,1,2}
{
    \draw[blue,thick] (-8,{-4-\d}) -- (8,{4-\d});
}

% Central hyperplane x2=x3
\draw[blue,very thick] (-8,-4) -- (8,4);
\node[text=blue,rotate=26.565] at (9.2,4.6) {$x_2=x_3$};

%---------------------------------------------------
% Hyperplanes x1-x3=e
%---------------------------------------------------
\foreach \e in {-2,-1,1,2,3}
{
    \draw[greenline,thick] (-8,{4-\e}) -- (8,{-4-\e});
}

% Central hyperplane x1=x3
\draw[greenline,very thick] (-8,4) -- (8,-4);
\node[text=greenline,rotate=-26.565] at (-9.2,4.6) {$x_1=x_3$};

%---------------------------------------------------
% Origin
%---------------------------------------------------
\fill (0,0) circle (3pt);

%---------------------------------------------------
% Chamber labels
%---------------------------------------------------
\node at (-5.2,7) {$x_2>x_1>x_3$};

\node at (5.2,7) {$x_1>x_2>x_3$};

\node at (8.2,0.55) {$x_1>x_3>x_2$};

\node at (-5.2,-6.8) {$x_3>x_2>x_1$};

\node at (-8.2,-0.55) {$x_2>x_3>x_1$};

\node at (5.2,-6.9) {$x_3>x_1>x_2$};

\node at (-0.3,0.6) {A};
%\node at (2.5, 1.8) {B};
        
\end{tikzpicture}
\caption{Asymmetric $\mathbf{c}$-Catalan arrangement for $\mathbf{c}=(4,2,1)$}
\label{fig:catalan_regions}
\end{figure}
\end{example}
\begin{note}
    If we take $c_i=k$ for all $i\in [n]$ then $\mathcal{A}_{\mathrm{Cat}}^{\mathbf{c}}$  gives the $k$-Catalan arrangement (see \ref{k-catalan}).
\end{note}

%\begin{note}\label{c to m cat}
    % We now construct a bijection between the regions of the  asymmetric $\mathbf{c}$-Catalan arrangement and the $\mathbf{c}$-labelled Dyck paths (see \ref{c-dyckpath}). Therefore by Theorem \ref{m dyck}, number of $\mathbf{c}$-labelled Dyck paths is 
    % $$
    % \frac{n!}{C+1}\binom{C+n}{n}.
    % $$
    % where $\displaystyle C=\sum_{i=1}^nc_i$.
%\end{note}

We now establish a bijection between the regions of the asymmetric $\mathbf{c}$-Catalan arrangement and the $\mathbf{c}$-labelled Dyck paths. The construction is carried out in two steps. First, we define a bijection $\varphi_{\mathbf{c}}$ between the regions of the asymmetric $\mathbf{c}$-Catalan arrangement and the Catalan $\mathbf{c}$-ordered words. Next, we construct a bijection $\psi_{\mathbf{c}}$ from the Catalan $\mathbf{c}$-ordered words to the $\mathbf{c}$-labelled Dyck paths. Consequently, the composition $(\psi_{\mathbf{c}}\circ\varphi_{\mathbf{c}})$ gives the desired bijection between the regions of the asymmetric $\mathbf{c}$-Catalan arrangement and the $\mathbf{c}$-labelled Dyck paths.

\begin{note}\label{theo_cat_sol}
    Let $\mathbf{c}=2\mathbf{m}$, that is, $c_i=2m_i$ for all $i\in[n]$. Then the arrangement $\mathcal{A}_{\mathrm{Cat}}^{2\mathbf{m}}$ coincides with the arrangement considered in Equation \ref{theo_cat}, which was originally proposed by Theo Douvropoulos. Consequently, the number of regions of $\mathcal{A}_{\mathrm{Cat}}^{2\mathbf{m}}$ is equal to the number of $2\mathbf{m}$-labelled Dyck paths.
\end{note}

\subsection{Bijection between regions of asymmetric $\mathbf{c}$-Catalan arrangement and the  Catalan $\mathbf{c}$-ordered word}

\begin{definition}\label{defn:gen_cat_m_ord_seqn}
A \textit{Catalan $\mathbf{c}$-ordered
word} is an ordering of the set of all  \textit{alphabets}
$$
\begin{aligned} 
    A(\mathbf{c})=&\{a_i^{(t)}\mid i\in [n],\ t\in[-m_i,m_i-\mathds{1}_I(i)]\}\\
    =&\{a_i^{(t)}\mid i\notin I,\ t\in[-m_i,m_i]\} \cup \{a_i^{(t)}\mid i\in I,\ t\in[-m_i,m_i-1]\},
\end{aligned}
$$
satisfying the following conditions:
\begin{itemize}
    \item For every $i\in[n]$ and every $t$ such that
    $a_i^{(t-1)}$ and $a_i^{(t)}$ belong to $A(\mathbf{c})$,
    the letter $a_i^{(t-1)}$ appears before $a_i^{(t)}$.

    \item For all $i,j\in[n]$, and all $t_i,t_j$ such that
    $a_i^{(t_i-1)},a_i^{(t_i)},a_j^{(t_j-1)},a_j^{(t_j)}$ belong to
    $A(\mathbf{c})$, the relative order of
    $a_i^{(t_i)}$ and $a_j^{(t_j)}$ is the same as the relative order of
    $a_i^{(t_i-1)}$ and $a_j^{(t_j-1)}$; equivalently,
    $a_i^{(t_i)}$ appears before $a_j^{(t_j)}$ if and only if $a_i^{(t_i-1)}$ appears before $a_j^{(t_j-1)}$.
\end{itemize}
The set of all Catalan $\mathbf{c}$-ordered words is denoted by $\mathcal{W}_C(\mathbf{c})$.
\end{definition}

We now define the map $\varphi_{\mathbf{c}}$ from $\mathcal{R}_{C}(\mathbf{c})$, the set of regions of the asymmetric $\mathbf{c}$-Catalan arrangement, to $\mathcal{W}_{C}(\mathbf{c})$, the set of Catalan $\mathbf{c}$-ordered words. 

\medskip

\noindent \textbf{The map $\varphi_{\mathbf{c}}:\mathcal{R}_{C}(\mathbf{c})\rightarrow \mathcal{W}_{C}(\mathbf{c})$:} Let $\Delta\in\mathcal{R}_{C}(\mathbf{c})$, and choose a point $(x_1,\dots,x_n)\in \Delta$. We construct the Catalan $\mathbf{c}$-ordered word corresponding to $\Delta$ as follows. 

\begin{enumerate}
\item For each pair of distinct indices $i,j\in[n]$, exactly one of the
following holds, according to the membership of $i$ and $j$ in $I$. There
\begin{itemize}
\item there exists an integer $k\in[-m_i-m_j+\mathds{1}_{I}(i),m_i+m_j-\mathds{1}_{I}(j)-1]$ such that $k<x_i-x_j<k+1$,
\item $x_i-x_j<-(m_i+m_j)+\mathds{1}_{I}(i)$, or
\item $x_i-x_j>m_i+m_j-\mathds{1}_{I}(j)-1$.
\end{itemize}

Consequently, the point $(x_1,\dots,x_n)$ determines a total ordering of the set 
$$
\{x_i+t\mid i\in I,\, t\in[-m_i,m_i]\}\cup \{x_i+t\mid i\notin I,\, t\in[-m_i,m_i-1]\}.
$$ 
Since every point in the region $\Delta$ satisfies the same defining inequalities, this ordering is independent of the choice of the point in $\Delta$.

\item If the $r$-th element in the above ordering is $x_i+t$, then define $w_r=a_i^{(t)}$. The map $\varphi_{\mathbf{c}}$ corresponds the region $\Delta$ to the Catalan $\mathbf{c}$-ordered word $\mathbf{w}=w_1w_2\cdots w_{C+n}$.
\end{enumerate}

Before we prove that $\varphi_{\mathbf{c}}$ is a bijection, we explain the above construction through the following example.

\begin{example}
    For $\mathbf{c}=(4,2,1)$, the asymmetric $\mathbf{c}$-Catalan arrangement  $$\mathcal{A}_{\mathrm{Cat}}^{\mathbf{c}}=\{ x_1-x_2 \in [-3,3],\quad \,x_1-x_3 \in [-3,2],\quad x_2-x_3 \in [-2,1]\} $$ is shown in the Figure \ref{fig:catalan_regions}.

\noindent Consider the region marked as A, we have $$-1<x_1-x_2<0, \quad 0<x_1-x_3<1,\quad  0<x_2-x_3<1.$$ By rewriting we get,
$$x_1-2<x_3-1<x_1-1<x_2-1<x_3<x_1<x_2<x_1+1<x_2+1<x_1+2.$$ 
This gives Catalan $\mathbf{c}$-ordered word $$a_1^{(-2)}a_3^{(-1)}a_1^{(-1)}a_2^{(-1)}a_3^{(0)}a_1^{(0)}a_2^{(0)}a_1^{(1)}a_2^{(1)}a_1^{(2)}.$$ 
% \noindent Consider the region marked as B, we have $$2<x_1-x_2<3, \quad 3<x_1-x_3,\quad  0<x_2-x_3<1 .$$ By rewriting we get, 
% $$x_3-1<x_3<x_2-1<x_2<x_2+1<x_1-2<x_1-1<x_1<x_1+1<x_1+2.$$
% This gives Catalan $\mathbf{c}$-ordered word
% $$a_1^{(-2)}a_1^{(-1)}a_3^{(-1)}a_2^{(-1)}a_3^{(0)}a_1^{(0)}a_2^{(0)}a_1^{(1)}a_2^{(1)}a_1^{(2)}.$$ 
Similarly, we can do for other regions. Every region corresponds to a unique Catalan $\mathbf{c}$-ordered word.
\end{example}

\begin{theorem}\label{thm:cat_region_and_words}
The map $\varphi_{\mathbf{c}}$ is a bijection between
$\mathcal{R}_{C}(\mathbf{c})$ and $\mathcal{W}_{C}(\mathbf{c})$.
\end{theorem}

\begin{proof}
Since every point in a region $\Delta$ satisfies the same defining inequalities, the ordering used to define $\varphi_{\mathbf{c}}(\Delta)$ is independent of the choice of a point in $\Delta$. Thus, $\varphi_{\mathbf{c}}$ is well defined.

We first prove that $\varphi_{\mathbf{c}}$ is injective. Suppose that
$\Delta_1,\Delta_2\in\mathcal{R}_{C}(\mathbf{c})$ are distinct regions.
Then there exist $s\geq 0$ and a hyperplane $x_i-x_j=s$ separating
$\Delta_1$ and $\Delta_2$. Without loss of generality, assume that
$x_i-x_j<s$ on $\Delta_1$ and $x_i-x_j>s$ on $\Delta_2$.

\noindent\textbf{Case 1:} If $s\leq m_j-\mathds{1}_I(j)$, then
$a_i^{(0)}$ appears before $a_j^{(s)}$ in
$\varphi_{\mathbf{c}}(\Delta_1)$, whereas $a_i^{(0)}$ appears after
$a_j^{(s)}$ in $\varphi_{\mathbf{c}}(\Delta_2)$.

\medskip

\noindent\textbf{Case 2:} If $s>m_j-\mathds{1}_I(j)$, then
$a_i^{(-s+m_j-\mathds{1}_I(j))}$ appears before $a_j^{(m_j-\mathds{1}_I(j))}$ in
$\varphi_{\mathbf{c}}(\Delta_1)$, whereas $a_i^{(-s+m_j-\mathds{1}_I(j))}$ appears after $a_j^{(m_j-\mathds{1}_I(j))}$ in $\varphi_{\mathbf{c}}(\Delta_2)$. If $s\leq m_j-\mathds{1}_I(j)$, then $a_i^{(0)}$ appears before $a_j^{(s)}$ in $\varphi_{\mathbf{c}}(\Delta_1)$, whereas $a_i^{(0)}$ appears after $a_j^{(s)}$ in $\varphi_{\mathbf{c}}(\Delta_2)$.

\medskip
\noindent Thus, in both cases,
$$
\varphi_{\mathbf{c}}(\Delta_1) \neq \varphi_{\mathbf{c}}(\Delta_2).
$$
Therefore, distinct regions have distinct images under $\varphi_{\mathbf{c}}$, and hence $\varphi_{\mathbf{c}}$ is injective.

For the surjectivity of the map $\varphi_{\mathbf{c}}$, take any Catalan $\mathbf{c}$-ordered word of
$$
\mathbf{w}=w_1w_2\dots w_{C+n}\in\mathcal{W}_{C}(\mathbf{c}).
$$
We recursively construct a point $x=(x_1,x_2,\dots,x_n)\in\mathbb{R}^n$ such that the region $\Delta$ containing $x$ satisfies $\varphi_{\mathbf{c}}(\Delta)=\mathbf{w}$. To carry out the construction, we introduce an auxiliary sequence $z_1,\dots,z_{C+n}$ of real numbers, where $z_p$ denotes the value assigned to the $p$-th letter of $\mathbf{w}$. We need to follow these steps until all the $x_i$'s have been assigned a value. Begin by setting $z_1=0$. If $w_1=a_i^{(-m_i)}$, then define $x_i=0$. For each $p\in[2,C+n]$, proceed as follows. 

\noindent\textbf{Case 1:} If $w_p=a_i^{(-m_i)}$, define
$$
z_p=z_{p-1}+\frac{1}{n+1},
$$
and set $x_i=z_p$.

\medskip
\noindent\textbf{Case 2:} If $w_p=a_i^{(t)}$ with $t\in[-m_i+1,m_i-\mathds{1}_I(i)]$, define
$$
z_p=x_i+(t+m_i).
$$

\medskip
\noindent Then $\varphi_{\mathbf{c}}(\Delta)=\mathbf{w}$, proving that the map $\varphi_{\mathbf{c}}$ is surjective. Hence, $\varphi_{\mathbf{c}}$ is a bijection.
\end{proof}

\begin{note}\label{same size R, C}
    As we established a bijection between $\mathcal{R}_{C}(\mathbf{c})$ and $\mathcal{W}_{C}(\mathbf{c})$, the cardinality of both the sets are same. Therefore, by Theorem \ref{char-catalan} we have, the size of $\mathcal{W}_{C}(\mathbf{c})$ is  
    $$
    \frac{n!}{C + 1} \binom{C+n}{n}
    $$ 
    where $C=\displaystyle \sum_{i=1}^n c_i$. 
\end{note}

\subsection{Bijection between the  Catalan $\mathbf{c}$-ordered word and $\mathbf{c}$-labelled Dyck paths}

We now define a map from  $\mathcal{W}_{C}(\mathbf{c})$, the set of all Catalan $\mathbf{c}$-ordered words to $\mathcal{D}(\mathbf{c})$, the set of all $\mathbf{c}$-labelled Dyck paths. 

\medskip

\noindent \textbf{The map $\psi_{\mathbf{c}}:\mathcal{W}_{C}({\mathbf{c}})\rightarrow \mathcal{D}({\mathbf{c}})$:} Let $\mathbf{w}=w_1w_2\dots w_{C+n}\in \mathcal{W}_{C}(\mathbf{c})$ be a Catalan $\mathbf{c}$-ordered word. Construct a path $P_w$ from $(0,0)$ to $(C+n,0)$ following the below steps: 
\begin{enumerate}
    \item if $w_i=a_j^{(-m_j)}$ the $i$-th step in $P_w$ is a up step $U_j=(1,c_j)$.
    \item if $w_i=a_j^{(k)}$ where $k\in [-m_j+1, m_j]$ the $i$-th step in $P_w$ is $(1,-1)$. 
\end{enumerate}
We illustrate this map by the following example.

\begin{example}\label{example:catalan_word_to_dyck_path}
Let $\mathbf{c}=(4,2,1)$, hence $\mathbf{m}=(m_1,m_2,m_3)=(2,1,1)$. For the Catalan $\mathbf{c}$-ordered  word
$$a_1^{(-2)}a_3^{(-1)}a_1^{(-1)}a_2^{(-1)}a_3^{(0)}a_1^{(0)}a_2^{(0)}a_1^{(1)}a_2^{(1)}a_1^{(2)},$$ 
the path $P_w$ is  $
U_1\,U_3\,D\,U_2\,D\,D\,D\,D\,D\,D.$
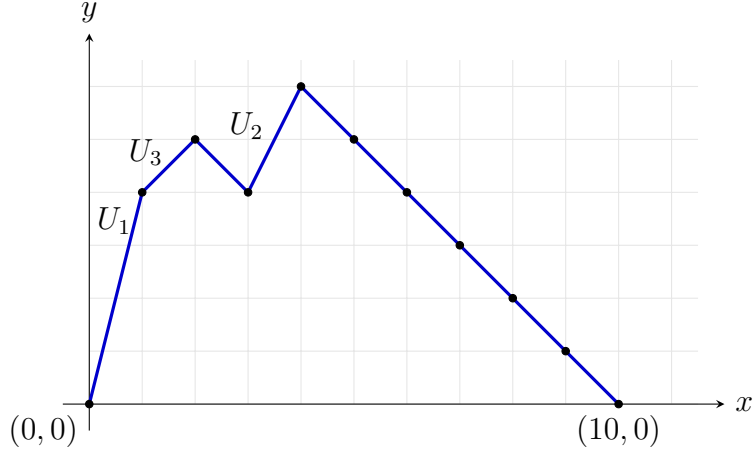
\vspace{3mm}
\begin{center}
\begin{tikzpicture}[scale=0.7]

    % Optional: Faint grid for readability
    \draw[thin, gray!20, step=1] (0,0) grid (11.5, 6.5);

    % Axes
    \draw[->, >=stealth, thin] (-0.5,0) -- (12,0) node[right] {$x$};
    \draw[->, >=stealth, thin] (0,-0.5) -- (0,7) node[above] {$y$};

    % Continuous Path (Thicker and distinctly colored)
    \draw[blue!80!black, line width=1.2pt] 
        (0,0) -- (1,4) -- (2,5) -- (3,4) -- (4,6) -- (5,5) 
        -- (6,4) -- (7,3) -- (8,2) -- (9,1) -- (10,0);

    % Vertices (Drawn after the path so they sit neatly on top)
    \foreach \p in {(0,0), (1,4), (2,5), (3,4), (4,6), (5,5),(6,4), (7,3), (8,2), (9,1), (10,0)}
        \filldraw[black] \p circle (2pt);

    % Step Labels (Anchored relative to the midpoints of the up-steps)
    \node[above left] at (1, 3) {$U_1$};
    \node[above left] at (1.6, 4.3) {$U_3$};
    \node[above left] at (3.5, 4.8) {$U_2$};

    % Coordinate Labels (Safely anchored to the exact nodes)
    \node[below left] at (0,0) {$(0,0)$};
    \node[below] at (10,0) {$(10,0)$};

\end{tikzpicture}
\end{center}
\end{example}

By Theorem \ref{m dyck} and Note \ref{same size R, C} we have  the number of Catalan $\mathbf{c}$-ordered words is equal to  the number of $\mathbf{c}$-labelled Dyck paths, i.e., 
$$
|\mathcal{W}_{C}(\mathbf{c})|=|\mathcal{D}(\mathbf{c})|= \frac{n!}{C + 1} \binom{C+n}{n},
$$ 
where $C=\displaystyle \sum_{i=1}^n c_i$. Therefore, to prove that the map $\psi_{\mathbf{c}}:\mathcal{W}_{C}(\mathbf{c})\rightarrow\mathcal{D}(\mathbf{c})$, defined by $w\mapsto P_w$, is a bijection, it suffices to show that $\psi_{\mathbf{c}}$ is injective.

\begin{theorem}\label{thm:word_and_dyck}
The map $\psi_{\mathbf{c}}$ is a bijection between $\mathcal{W}_{C}(\mathbf{c})$ and $\mathcal{D}(\mathbf{c})$.
\end{theorem}

\begin{proof}
We first show that the map is well defined. Let $\mathbf{w}=w_1w_2\cdots w_{C+n}\in\mathcal{W}_{C}(\mathbf{c})$. We must show that the lattice path constructed by $\psi_{\mathbf{c}}(w)$ never goes below the $x$-axis. 

For each $j\in[C+n]$, let $\mathbf{w}^{(j)}=w_1w_2\cdots w_j$ denote the prefix of $\mathbf{w}$ of length $j$. Consider an alphabet $a_i$ occurring in $\mathbf{w}^{(j)}$. There exists $r\in[-m_i,(m_i-\mathds{1}_I(i))]$ such that $a_i^{(-m_i)},\ldots,a_i^{(r)}$ occur in $w^{(j)}$, while $a_i^{(r+1)},\ldots,a_i^{(m_i-\mathds{1}_I(i))}$ do not. Hence, the net contribution of $a_i$ to the ordinate after reading $\mathbf{w}^{(j)}$ is $m_i-\mathds{1}_I(i)-r\geq0$. Thus, every alphabet occurring in $\mathbf{w}^{(j)}$ contributes non-negatively to the ordinate, and so the ordinate after reading $\mathbf{w}^{(j)}$ is non-negative. Since this holds for every $j\in[C+n]$, the path constructed by $\psi_{\mathbf{c}}(\mathbf{w})$ never goes below the $x$-axis. Hence, $\psi_{\mathbf{c}}$ is well defined.

We now prove that the map $\psi_{\mathbf{c}}$ is injective. Let $\mathbf{w},\mathbf{w}'\in\mathcal{W}_{C}(\mathbf{c})$, and let $\ell\in[C+n]$ be the first position from the left at which $\mathbf{w}$ and $\mathbf{w}'$ differ. Suppose that the $\ell$-th letters of $\mathbf{w}$ and $\mathbf{w}'$ are $a_{i_1}^{(l_1)}$ and $a_{i_2}^{(l_2)}$, respectively.

We claim that $l_1=-k_{i_1}$ or $l_2=-k_{i_2}$. Suppose, to the contrary, that neither equality holds. Then the alphabets with indices $i_1$ and $i_2$ must both have appeared before the $\ell$-th position in $\mathbf{w}$ and $\mathbf{w}'$, respectively. Since $\mathbf{w}$ and $\mathbf{w}'$ both satisfy Definition~\ref{defn:gen_cat_m_ord_seqn}, it follows that $a_{i_1}^{(l_1)}=a_{i_2}^{(l_2)}$, contradicting the choice of $\ell$. Hence, the claim follows.

Without loss of generality, suppose that $l_1=-k_{i_1}$. Then the $\ell$-th step of the $c$-labelled Dyck path corresponding to $\psi_{\mathbf{c}}(\mathbf{w})$ is the up-step labelled $U_{i_1}$. Since the $\ell$-th letters of $\mathbf{w}$ and $\mathbf{w}'$ are distinct, the $\ell$-th step of the path corresponding to $\psi_{\mathbf{c}}(\mathbf{w}')$ is different from $U_{i_1}$. Therefore, $\psi_{\mathbf{c}}(\mathbf{w})\neq\psi_{\mathbf{c}}(\mathbf{w}')$. Thus, $\psi_{\mathbf{c}}$ is injective. Hence, $\psi_{\mathbf{c}}$ is bijective.
\end{proof}

\noindent We now define the inverse of the map ${\psi}_{\mathbf{c}}$. We omit the proof that it is indeed the inverse of $\psi_{\mathbf{c}}$.

\medskip

\noindent \textbf{The map $(\psi_{\mathbf{c}})^{-1}:\mathcal{D}({\mathbf{c}}) \rightarrow \mathcal{W}_C({\mathbf{c}})$:}
For each step $i\in[C+n]$, we maintain a tuple $T_i$ of indices, referred to as the
\textit{cyclic index tuple}. We regard each $T_i$ as a cyclic tuple; that is, the index following the last element of $T_i$ is the first element of $T_i$.

Suppose that the first step of the $\mathbf{c}$-labelled Dyck path is the up-step labelled by $U_j$. We initialize the construction by setting
$$
w_1=a_j^{(-m_j)},\qquad T_1=(j).
$$
Now, suppose that the first $(i-1)$ steps have been processed and that the last letter added is $w_{i-1}=a_{j_1}^{(k_1)}$. Let $T_{i-1}$ denote the corresponding cyclic index tuple. We define the next letter $w_i$ and the updated tuple $T_i$ according to the $i$-th step of the $\mathbf{c}$-labelled Dyck path as follows.

\begin{enumerate}
    \item If the $i$-th step is an up-step labelled by $U_{j_2}$,
    then we set $w_i=a_{j_2}^{(-m_{j_2})}$. The tuple $T_i$ is obtained by inserting $j_2$ immediately after $j_1$ in
    $T_{i-1}$.

    \item If the $i$-th step is a down-step $(1,-1)$, let $j_2$ be the index appearing immediately after $j_1$ in the cyclic tuple $T_{i-1}$. Let $k_2$ be determined by the condition that $a_{j_2}^{(k_2-1)}$ is the most recently used element of the form $a_{j_2}^{(t)}$ among $w_1,w_2,\dots,w_{i-1}$. We then set $w_i=a_{j_2}^{(k_2)}$. If $k_1<m_{j_1}-\mathds{1}_I(j_1)$, set $T_i=T_{i-1}$, i.e. $T_i$ is unchanged. Otherwise, if $k_1=m_{j_1}-\mathds{1}_I(j_1)$, remove $j_1$ from $T_{i-1}$ to obtain $T_i$.
\end{enumerate}

Let us illustrate the construction of this map with an example.

\begin{example}
We now explain the inverse map using an example. Consider the $\mathbf c$-labelled Dyck path given in Example \ref{example:catalan_word_to_dyck_path},
$$
U_1\,U_3\,D\,U_2\,D\,D\,D\,D\,D\,D.$$
We now apply the above algorithm to construct the corresponding Catalan
$\mathbf c$-ordered word. 

\begin{itemize}
\item The first step is the up-step labelled by $U_1$,
and hence we set $w_1=a_1^{(-2)}$ and  $T_1=(1)$.
\item The second step is again an up-step labelled by $U_3$. We insert the new index $3$
immediately after $1$ in the cyclic tuple and set
$w_2=a_3^{(-1)}$, and $T_2=(1,3)$.
\item The third step is a down-step. The index immediately following $3$ in the cyclic tuple $T_2$ is again $1$. Since the most recently used letter of type $1$ is $a_1^{(-2)}$, we obtain $w_3=a_1^{(-1)}$, $T_3$ is unchanged, i.e., $T_3=(1,3)$.
\item The fourth step is an up-step labelled by $U_2$. We insert the new index $2$
immediately after $3$ in the cyclic tuple and set
$w_4=a_2^{(-1)}$, and $T_4=(1,3,2)$.
\item The fifth step is a down-step. The index following $2$ in $T_4$ is $3$. The most recently used letter of type $3$ is $a_3^{(-1)}$, and therefore $w_5=a_3^{(0)}$, we get $T_5=(1,3,2)$. 

\item The sixth step is again a down-step. The index following $3$ is $1$ in $T_5$.
Since $a_1^{(-1)}$ is the most recently used letter of type $1$, we obtain
$w_6=a_1^{(0)}$, and removing $3$ we get  $T_6=(1,2)$.

\item The seventh step is again a down-step. The index following $1$ in $T_6$ is $2$.
Since $a_2^{(-1)}$ is the most recently used letter of type $2$, we obtain
$w_7=a_2^{(0)}$, and $T_7=(1,2)$.

\item Similarly the remaining down-steps give
$$
\begin{aligned}
&w_8=a_1^{(1)}, &T_8=(1,2),\\
&w_{9}=a_2^{(1)}, &T_{9}=(1,2),\\
&w_{10}=a_1^{(2)}, &T_{10}=(1) .
\end{aligned}
$$

\end{itemize}

Hence the corresponding Catalan $\mathbf c$-ordered word is
$$a_1^{(-2)}a_3^{(-1)}a_1^{(-1)}a_2^{(-1)}a_3^{(0)}a_1^{(0)}a_2^{(0)}a_1^{(1)}a_2^{(1)}a_1^{(2)},$$ as in Example \ref{example:catalan_word_to_dyck_path}.
\end{example}

\noindent Here we give a rough sketch of the proof that $\left(\psi_{\mathbf{c}}\right)^{-1}$ is indeed the inverse of $\psi_{\mathbf{c}}$.

\begin{proof}[Sketch of the proof that $(\psi_{\mathbf{c}})^{-1}$ is inverse of $\psi_{\mathbf{c}}$]
Let $P_w\in\mathcal{D}(\mathbf{c})$ be an arbitrary $\mathbf{c}$-labelled
Dyck path. We construct a word
$\mathbf{w}=w_1w_2\cdots w_{C+n}$ using the above algorithm. The cyclic
index tuple keeps track of the indices that are currently active. When an
up-step $U_j$ is encountered, we add the letter
$a_j^{(-m_j)}$ and insert $j$ into the cyclic tuple. When a down-step is
encountered, the next active index $j_2$ is uniquely determined by the
cyclic order, and the corresponding superscript $k_2$ is uniquely
determined by the most recently used letter associated with $j_2$. The
update rule removes an index precisely when its last allowable letter has
been used. Hence, at every step, the next letter and the corresponding
cyclic tuple are uniquely determined, so the construction produces a
well-defined word $\mathbf{w}\in\mathcal{W}_C(\mathbf{c})$.

By construction, an up-step $U_j$ of $P_w$ produces the letter
$a_j^{(-m_j)}$, while every down-step produces a letter
$a_j^{(k)}$ with $k\in[-m_j+1,m_j-\mathds{1}_{I}(j)]$. Therefore, applying
$\psi_{\mathbf{c}}$ to the word constructed above reproduces exactly the
same sequence of steps as $P_w$. Thus,
$$
\psi_{\mathbf{c}}(\mathbf{w})=P_w.
$$
and, since ${\psi}_{\mathbf c}$ is bijective, the constructed map
is precisely $({\psi}_{\mathbf c})^{-1}$.
\end{proof}
% The next theorem gives a precise formulation of Theorem \ref{main_cat0}.

We now establish the first main result of this paper: an explicit bijection between the set of regions of the asymmetric $\mathbf{c}$-Catalan arrangement, $\mathcal{R}_{C}(\mathbf{c})$, and the set of $\mathbf{c}$-labelled Dyck paths, $\mathcal{D}(\mathbf{c})$.

\begin{proof}[\textbf{Proof of Theorem \ref{main-bijection-cat1}}]
By Theorems \ref{thm:cat_region_and_words} and \ref{thm:word_and_dyck}, both $\varphi_{\mathbf{c}}$ and $\psi_{\mathbf{c}}$ are bijections. Therefore, their composition $\psi_{\mathbf{c}}\circ\varphi_{\mathbf{c}}$ is also a bijection  between $\mathcal{R}_{C}(\mathbf{c})$ and $\mathcal{D}(\mathbf{c})$. We provide a commutative diagram for more clarity.

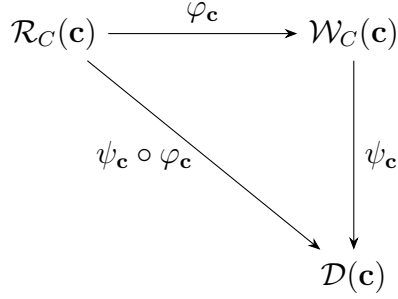
\begin{figure}[H]
\begin{center}
    \begin{tikzpicture}
    [>=Stealth,node distance=2.5cm and 2.5cm]

\node (R) {$\mathcal{R}_{C}(\mathbf{c})$};
\node[right=of R] (A) {$\mathcal{W}_{C}(\mathbf{c})$};
\node[below=of A] (D) {$\mathcal{D}(\mathbf{c})$};

% Horizontal arrow
\draw[->] (R) -- node[above] {$\varphi_{\mathbf{c}}$} (A);

% Vertical arrow
\draw[->] (A) -- node[right] {$\psi_{\mathbf{c}}$} (D);

% Diagonal arrow
%\draw[->] (R) -- (D);
\draw[->] (R) -- node[below,left] {$\psi_{\mathbf{c}}\circ\varphi_{\mathbf{c}}$} (D);

% Composition label
%\node[left=1.2cm of D] {$\psi_m\circ\phi_m$};

\end{tikzpicture}
\end{center}
\caption{Commutative diagram of bijection}
\end{figure}
\end{proof}
%\noindent\rule{\textwidth}{0.4pt}

\section{Asymmetric $\mathbf c$-Shi Arrangement}\label{sec:ASA}

We now turn our attention to the asymmetric $\mathbf c$-Shi arrangement. This naturally extends the previously studied $k$-Shi arrangements.
Our main goal is to give a bijective description of the regions of this arrangement. To this end, we establish a bijection between its regions and the $\mathbf{c}$-building functions introduced in Section \ref{sec: c-building}.

Given a tuple $\mathbf{c}=(c_1,c_2,\dots,c_n)$, we will show that the regions of the asymmetric $\mathbf{c}$-Shi
arrangement are in bijection with the $\mathbf{c}$-building functions.
Recall that the asymmetric $\mathbf{c}$-Shi arrangement $\mathcal{A}_{\mathrm{Shi}}^{\mathbf{c}}$ is the hyperplane arrangement defined in Definition~\ref{shi-kind}.

For each $i\in[n]$, define $\mathbf{m}=(m_1,m_2,\ldots , m_n)$ such that
$$
m_i=
\begin{cases}
\dfrac{c_i+1}{2}, & \text{if } c_i \text{ is odd},\\[6pt]
\dfrac{c_i}{2}, & \text{if } c_i \text{ is even}.
\end{cases}
$$ and let $I=\{i\in[n]\mid c_i\text{ is odd}\}$. The asymmetric $\mathbf{c}$-Shi arrangement is the hyperplane arrangement
$$
\mathcal{A}_{\mathrm{Shi}}^{\mathbf{c}}=\{ x_i-x_j \in [-m_{i}-m_{j}+1+\mathds{1}_{I}(i),m_i+m_j-\mathds{1}_{I}(j)]\}.
$$
We denote the set of regions of asymmetric $\mathbf{c}$-Shi arrangement by $\mathcal{R}_{S}(\mathbf{c})$.

% Throughout this section, for each $i\in[n]$, define
% $$
% m_i=
% \begin{cases}
% \dfrac{c_i+1}{2}, & \text{if } c_i \text{ is odd},\\[6pt]
% \dfrac{c_i}{2}, & \text{if } c_i \text{ is even},
% \end{cases}
% $$
% and let 
% $$
% I=\{i\in[n]\mid c_i\text{ is even}\}.
% $$

% \begin{note}\label{c to m shi}
% We will construct a bijection between the regions of the  asymmetric $\mathbf{c}$-Shi arrangement and the $\mathbf{c}$-building functions (see \ref{c-dyckpath}). Therefore by Theorem \ref{count_building-func}, number of $\mathbf{c}$-building functions is 
% $$(1+C)^{n-1}.$$
% where $\displaystyle C=\sum_{i=1}^nc_i$.
% \end{note}
\begin{example}\label{ex:shi_regions}
    Let us consider an example where $n=3$, $\mathbf{c}=(4,2,1)$, hence $\mathbf{m}=(2,1,1)$ and $I=\{3\}.$  Therefore 
\[
\mathcal{A}_{\mathrm{Shi}}^{\mathbf{c}}=\left\{
x_1-x_2 \in [-2,3],\quad 
x_1-x_3 \in[-2,2],\quad
x_2-x_3 \in[-1,1]\right\}.
\]
 We have total $64$ regions in Figure \ref{fig:shi_regions}. 
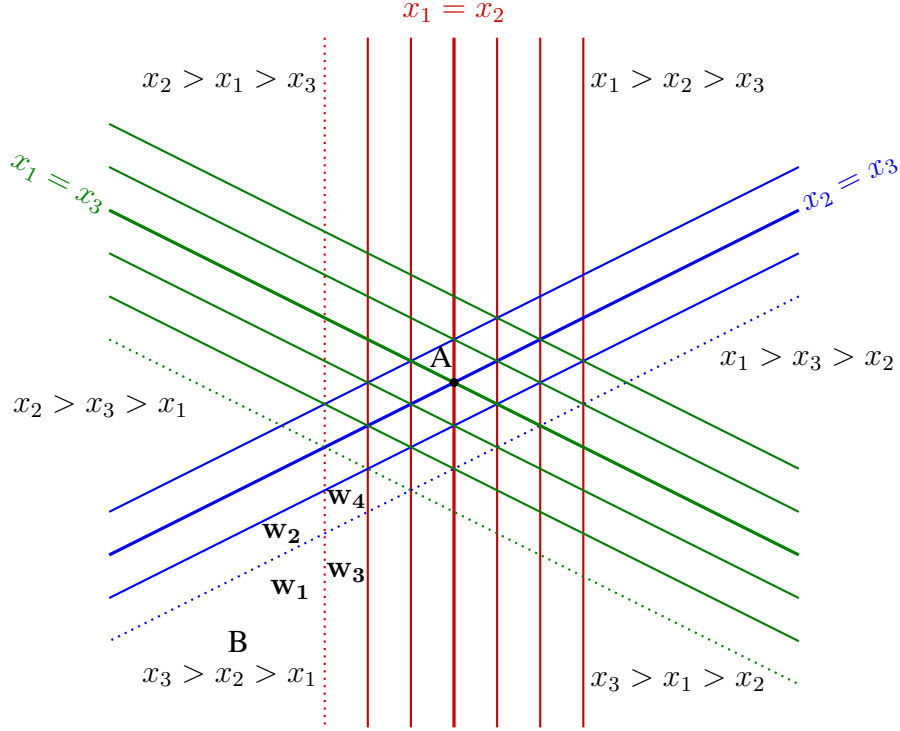
\begin{figure}[htbp]
\centering
\begin{tikzpicture}[xscale=1,yscale=1,scale=0.57]

\colorlet{redline}{red!80!black}
\colorlet{greenline}{green!50!black}

%---------------------------------------------------
% Hyperplanes x1-x2=c
%---------------------------------------------------
\foreach \c in {-2,-1,1,2,3}
{
    \draw[redline,thick] (\c,-8) -- (\c,8);
}

% Central hyperplane x1=x2
\draw[redline,very thick] (0,-8) -- (0,8);
\node[text=redline, above] at (0,8.1) {$x_1=x_2$};
\draw[redline,dotted, thick] (-3,8) -- (-3,-8);
%---------------------------------------------------
% Hyperplanes x2-x3=d
%---------------------------------------------------
\foreach \d in {-1,1}
{
    \draw[blue,thick] (-8,{-4-\d}) -- (8,{4-\d});
}

% Central hyperplane x2=x3
\draw[blue,very thick] (-8,-4) -- (8,4);
\node[text=blue,rotate=26.565] at (9.2,4.6) {$x_2=x_3$};
\draw[blue,dotted, thick] (8,2) -- (-8,-6);
%---------------------------------------------------
% Hyperplanes x1-x3=e
%---------------------------------------------------
\foreach \e in {-2,-1,1,2}
{
    \draw[greenline,thick] (-8,{4-\e}) -- (8,{-4-\e});
}

% Central hyperplane x1=x3
\draw[greenline,very thick] (-8,4) -- (8,-4);
\node[text=greenline,rotate=-26.565] at (-9.2,4.6) {$x_1=x_3$};
\draw[greenline,dotted, thick] (-8,1) -- (8,-7);
%---------------------------------------------------
% Origin
%---------------------------------------------------
\fill (0,0) circle (3pt);

%---------------------------------------------------
% Chamber labels
%---------------------------------------------------
\node at (-5.2,7) {$x_2>x_1>x_3$};

\node at (5.2,7) {$x_1>x_2>x_3$};

\node at (8.2,0.55) {$x_1>x_3>x_2$};

\node at (-5.2,-6.8) {$x_3>x_2>x_1$};

\node at (-8.2,-0.55) {$x_2>x_3>x_1$};

\node at (5.2,-6.9) {$x_3>x_1>x_2$};

\node at (-0.3,0.6) {A};
%\node at (2.5, 1.8) {B};
 \node at (-5.0,-6.0) {B};
\node at (-3.8,-4.8) {$\mathbf{w_1}$};
\node at (-4.0,-3.5) {$\mathbf{w_2}$};
\node at (-2.5,-4.4) {$\mathbf{w_3}$};
\node at (-2.5,-2.7) {$\mathbf{w_4}$};       
\end{tikzpicture}
\caption{Asymmetric $\mathbf{c}$-Shi arrangement for $\mathbf{c}=(4,2,1)$}
\label{fig:shi_regions}
\end{figure}

\end{example}

\begin{note}
    If we take $c_i=k$ for all $i\in [n]$ then $\mathcal{A}_{\mathrm{Shi}}^{\mathbf{c}}$  gives the $k$-Shi arrangement (see \ref{k-shi}).
\end{note}

We now establish a bijection between the regions of the asymmetric $\mathbf{c}$-Shi arrangement and the $\mathbf{c}$-building functions. The construction is carried out in two steps. First, we define a bijection $\tilde\varphi_{\mathbf{c}}$ between the regions of the asymmetric $\mathbf{c}$-Shi arrangement and the Shi $\mathbf{c}$-ordered words. Next, we construct a bijection $\tilde\phi_{\mathbf{c}}$ from the Shi $\mathbf{c}$-ordered words to the $\mathbf{c}$-building functions. Consequently, the composition $(\tilde\phi_{\mathbf{c}}\circ\tilde\varphi_{\mathbf{c}})$ gives the desired bijection between the regions of the asymmetric $\mathbf{c}$-Shi arrangement and the $\mathbf{c}$-building functions.
\begin{note}\label{theo_shi_sol}
    Let $\mathbf{c}=2\mathbf{m}$, that is, $c_i=2m_i$ for all $i\in[n]$. Then the arrangement $\mathcal{A}_{\mathrm{Shi}}^{2\mathbf{m}}$ coincides with the arrangement considered in Equation \ref{theo_shi}, which was originally proposed by Theo Douvropoulos. Consequently, the number of regions of $\mathcal{A}_{\mathrm{Shi}}^{2\mathbf{m}}$ is equal to the number of $2\mathbf{m}$-building function.
\end{note}

\subsection{Bijection between regions of  asymmetric $\mathbf{c}$-Shi arrangement and Shi $\mathbf{c}$-ordered words}

\begin{definition}
Let $\mathbf{w}$ and $\mathbf{w}'$ be Catalan $\mathbf{c}$-ordered words. We say that $\mathbf{w}$ and $\mathbf{w}'$ are related by a \textit{Shi move} if $\mathbf{w}'$ is obtained from $\mathbf{w}$ by swapping two consecutive alphabets $a_i^{(m_i-\mathds{1}_I(i))}$ and $a_j^{(-m_j)}$ with $i<j$.

We say that $\mathbf{w}$ and $\mathbf{w}'$ are \textit{$\mathbf{c}$-Shi equivalent} if $\mathbf{w}'$ can be obtained from $\mathbf{w}$ by performing a finite sequence of Shi moves.
\end{definition}

\begin{example}\label{ex:word_shi_1}

Consider the region marked as $A$ in Figure~\ref{fig:shi_regions}. The corresponding Catalan $\mathbf{c}$-ordered word is
$$
a_1^{(-2)}a_3^{(-1)}a_1^{(-1)}a_2^{(-1)}a_3^{(0)}a_1^{(0)}a_2^{(0)}a_1^{(1)}a_2^{(1)}a_1^{(2)}.
$$ 
There is no possible Shi move for this region. 

Now consider the region marked as $B$ in Figure \ref{fig:shi_regions}. There are four regions in the asymmetric $\mathbf{c}$-Catalan arrangement that correspond to region $B$ in the asymmetric $\mathbf{c}$-Shi arrangement. The corresponding Catalan $\mathbf{c}$-ordered word are as follows:
\begin{itemize}
    \item $\mathbf{w_1}=a_1^{(-2)}  a_1^{(-1)} a_1^{(0)}  a_1^{(1)}  a_1^{(2)} a_2^{(-1)}  a_2^{(0)}  a_2^{(1)}  a_3^{(-1)} a_3^{(0)}$,

    \item
    $\mathbf{w_2}=a_1^{(-2)}  a_1^{(-1)} a_1^{(0)}  a_1^{(1)}  a_1^{(2)}
    a_2^{(-1)} a_2^{(0)}  a_3^{(-1)}   a_2^{(1)} a_3^{(0)}$,

    \item
    $\mathbf{w_3}=a_1^{(-2)}  a_1^{(-1)} a_1^{(0)}  a_1^{(1)} a_2^{(-1)}  a_1^{(2)}   a_2^{(0)}  a_2^{(1)}  a_3^{(-1)} a_3^{(0)}$,

    \item
    $\mathbf{w_4}=a_1^{(-2)}  a_1^{(-1)} a_1^{(0)}  a_1^{(1)} a_2^{(-1)}  a_1^{(2)}   a_2^{(0)}  a_3^{(-1)} a_2^{(1)}   a_3^{(0)}.$ 
\end{itemize}
$\mathbf{w_2}$ is obtained from $\mathbf{w_1}$ by a  Shi move that swaps $a_2^{(1)}$ and $a_3^{(-1)}$. Hence, $\mathbf{w_1}$ and $\mathbf{w_2}$ are $\mathbf{c}$-Shi equivalent. Next, $\mathbf{w_3}$ is obtained from $\mathbf{w_1}$ by swapping $a_1^{(2)}$ and $a_2^{(-1)}$, and $\mathbf{w_4}$ is obtained from $\mathbf{w_3}$ by swapping $a_2^{(1)}$ and $a_3^{(-1)}$. Therefore, $\mathbf{w_1}$, $\mathbf{w_2}$, $\mathbf{w_3}$, and $\mathbf{w_4}$ all belong to the same $\mathbf{c}$-Shi equivalence class. 
Hence, the words $\mathbf{w_1}$, $\mathbf{w_2}$, $\mathbf{w_3}$, and $\mathbf{w_4}$ represent the same region in the asymmetric $\mathbf c$-Shi arrangement, whereas each represents a distinct region in the asymmetric $\mathbf c$-Catalan arrangement.
\end{example}
The  Shi moves define an equivalence relation on the set $\mathcal{W}_{C}(\mathbf{c})$. We now prove that there is a bijective correspondence between the resulting equivalence classes and the regions of the asymmetric $\mathbf c$-Shi arrangement.

\begin{theorem}\label{shush2} 
Let $\mathbf{w}, \mathbf{w}' \in \mathcal{W}_{C}(\mathbf{c})$, and let $\Delta=\varphi_{\mathbf{c}}(\mathbf{w})$ and $\Delta'=\varphi_{\mathbf{c}}(\mathbf{w}')$ be the corresponding regions of the asymmetric $\mathbf{c}$-Catalan arrangement. Then $\mathbf{w}$ and $\mathbf{w}'$ are $\mathbf{c}$-Shi equivalent if and only if $\Delta$ and $\Delta'$ lie in the same region of the  asymmetric $\mathbf{c}$-Shi arrangement.
\end{theorem}

\begin{proof}
Observe that $\mathbf{w}$ and $\mathbf{w}'$ differ by a single Shi move if and only if the corresponding regions $\Delta$ and $\Delta'$ are adjacent and separated by one of the following hyperplanes. If $i<j$, the move swaps $a_i^{(m_i-\mathds{1}_I(i))}$ and $a_j^{(-m_j)}$. In this case, $\Delta$ and $\Delta'$ are separated by the hyperplane
$$
x_i-x_j=-m_i-m_j+\mathds{1}_I(i).
$$

It follows that $\mathbf{w}$ and $\mathbf{w}'$ are $\mathbf{c}$-Shi equivalent if and only if one can move from $\Delta$ to $\Delta'$ by crossing only hyperplanes of the form
$$
x_i-x_j=-m_i-m_j+\mathds{1}_I(i).
$$
where $i< j$.
Equivalently, $\mathbf{w}$ and $\mathbf{w}'$ are $\mathbf{c}$-Shi equivalent if and only if the regions $\Delta$ and $\Delta'$ lie in the same region of the  asymmetric $\mathbf{c}$-Shi arrangement.
\end{proof}

\begin{definition}
A \emph{Shi $\mathbf{c}$-ordered word} is an equivalence class of $\mathcal{W}_{C}(\mathbf{c})$ under the equivalence relation induced by the  Shi moves.

\noindent The set of all Shi $\mathbf{c}$-ordered words is denoted by $\mathcal{W}_{S}(\mathbf{c})$.
\end{definition}

The following theorem summarizes the results established in this section.

\begin{theorem}\label{thm:shi_region_and_words}
The induced map $\tilde{\varphi}_{\mathbf c}$ is a bijection between $\mathcal R_{S}(\mathbf c)$ and $\mathcal W_{S}(\mathbf c)$.
\end{theorem}

\begin{proof}
Since $\mathcal{W}_{S}(\mathbf{c})$ is the quotient of $\mathcal{W}_{C}(\mathbf{c})$ by the  Shi move equivalence relation, Theorem~\ref{shush2} implies that the map $\varphi_{\mathbf c}$ induces a well-defined map
$$
\tilde{\varphi}_{\mathbf c}:\mathcal R_{S}(\mathbf c)\to\mathcal W_{S}(\mathbf c).
$$
The same theorem further implies that $\tilde{\varphi}_{\mathbf c}$ is a bijection.
\end{proof}

\begin{note}\label{same size R, S}
    As we established a bijection between $\mathcal{R}_{S}(\mathbf{c})$ and $\mathcal{W}_{S}(\mathbf{c})$, hence cardinality of both the sets are same. Therefore by Theorem \ref{char-catalan} we have, the size of $\mathcal{W}_{S}(\mathbf{c})$ is  $(1 + C)^{n-1}$
    where $C=\displaystyle \sum_{i=1}^n c_i$. 
\end{note}

\subsection{Bijection between  Shi $\mathbf{c}$-ordered word and $\mathbf{c}$-building function}

In this section, we define a map $$\tilde{\phi}_{\mathbf{c}}:\mathcal{W}_{S}(\mathbf{c})\to\mathcal{B}(\mathbf{c}),$$ 
where $\mathcal{W}_{S}(\mathbf{c})$ denotes the set of Shi $\mathbf{c}$-ordered words and $\mathcal{B}(\mathbf{c})$ denotes the set of $\mathbf{c}$-building functions. We prove that $\tilde{\phi}_{\mathbf{c}}$ is a bijection. 
We begin by defining a map 
$$
\phi_{\mathbf{c}}:\mathcal{W}_{C}(\mathbf{c})\to\mathcal{B}(\mathbf{c}).
$$ 
We then show that $\phi_{\mathbf{c}}$ is constant on the $\mathbf{c}$-Shi equivalence classes and therefore induces a well-defined map $\tilde{\phi}_{\mathbf{c}}$, which we subsequently prove to be a bijection.

\medskip 

\noindent \textbf{The map $\phi_{\mathbf{c}}:\mathcal{W}_{C}(\mathbf{c})\to\mathcal{B}(\mathbf{c})$:} The map $\phi_{\mathbf{c}}$ is constructed in the following steps.

\begin{enumerate}
\item For each $j\in[n]$, define the sets $A_j$ and $A_j'$ as follow
$$
A_j=\left\{a_j^{(-m_j+1)},a_j^{(-m_j+2)},\ldots,a_j^{(m_j-\mathds{1}_I(j))}\right\},  A_j'=\left\{a_j^{(-m_j)},a_j^{(-m_j+1)},\ldots,a_j^{(m_j-\mathds{1}_I(j)-1)}\right\}.
$$

\item For each $i\in[n]$, define the set of \emph{eligible alphabets for $i$} by
$$
E_i=\left(\bigcup_{j<i}A_j'\right)\cup\left(\bigcup_{k>i}A_k\right).
$$
Thus, $E_i$ contains the alphabets from $A_j'$ for $j<i$ and the alphabets from $A_k$ for $k>i$.

\item Given $\mathbf{w}\in\mathcal{W}_{C}(\mathbf{c})$, consider the ordering of the alphabets in $\mathbf{w}$. For each $i\in[n]$, count the number of alphabets in $E_i$ that appear before $a_i^{(-m_i)}$ in this ordering, and denote this number by $\phi_i(\mathbf{w})$.

\item Define the map $\phi_{\mathbf{c}}$ by
$$
\phi_{\mathbf{c}}(\mathbf{w})=\bigl(\phi_1(\mathbf{w}),\phi_2(\mathbf{w}),\dots,\phi_n(\mathbf{w})\bigr).
$$
\end{enumerate}

We illustrate this map with help of an example.

\begin{example} For $\mathbf{c}=(4,2,1)$, i.e., $\mathbf{m}=(2,1,1)$, the sets of eligible alphabets are
$$
E_1=\{a_2^{(0)},a_2^{(1)},a_3^{(0)}\},
$$
$$
E_2=\{a_1^{(-2)},a_1^{(-1)},a_1^{(0)},a_1^{(1)},a_3^{(0)}\},
$$
and
$$
E_3=\{a_1^{(-2)},a_1^{(-1)},a_1^{(0)},a_1^{(1)},a_2^{(-1)},a_2^{(0)}\}.
$$
The region $A$ in Figure~\ref{fig:shi_regions} corresponds to the Shi $\mathbf{c}$-ordered word 
$$
\mathbf{w}=a_1^{(-2)}a_3^{(-1)}a_1^{(-1)}a_2^{(-1)}a_3^{(0)}a_1^{(0)}a_2^{(0)}a_1^{(1)}a_2^{(1)}a_1^{(2)},
$$
In this ordering, there is no eligible alphabet appearing before $a_1^{(-2)}$. 
The eligible alphabets appearing before $a_2^{(-1)}$ are
$$
a_1^{(-2)},a_1^{(-1)},
$$
while one eligible alphabet $a_1^{(-2)}$ appears before $a_3^{(-1)}$. Therefore,
$$
\phi_{\mathbf c}(\mathbf w)=(0,2,1).
$$

The equivalence class of Shi $\mathbf{c}$-ordered words corresponding to the region $B$ in Figure~\ref{fig:shi_regions} consists of the four words $\mathbf{w}_1,\mathbf{w}_2,\mathbf{w}_3,$ and $\mathbf{w}_4$ given in Example~\ref{ex:word_shi_1}. For each of these words, the relative positions of the eligible alphabets with respect to $a_1^{(-2)}$, $a_2^{(-1)}$, and $a_3^{(-1)}$ remain unchanged. In particular, there are no eligible alphabets before $a_1^{(-2)}$, four eligible alphabets before $a_2^{(-1)}$, and six eligible alphabets before $a_3^{(-1)}$. Hence,
$$
\phi_{\mathbf{c}}(\mathbf{w}_1)
=\phi_{\mathbf{c}}(\mathbf{w}_2)
=\phi_{\mathbf{c}}(\mathbf{w}_3)
=\phi_{\mathbf{c}}(\mathbf{w}_4)
=(0,4,6).
$$
\end{example}

\noindent We now show that the map $\phi_{\mathbf c}$ is well defined.

\begin{lemma}\label{lem:ph1_1_well_defined}
    The map $\phi_{\mathbf c}$ is well defined.
\end{lemma}

\begin{proof}
Let $\mathbf w\in\mathcal{W}_{C}(\mathbf c)$ be a Catalan $\mathbf c$-ordered word. We show that $\phi_{\mathbf c}(\mathbf w)$ is a valid building function.

Suppose that, when $\mathbf w$ is read from left to right, the elements $a_i^{(-m_i)}$, $i\in[n]$, appear in the order determined by a permutation
$\pi\in S_n$. We claim that the firms can be successfully mobilized in the order $\pi(1),\pi(2),\ldots,\pi(n)$. For the first step, there are no previously mobilized firms. Hence, $\phi_{\pi(1)}(\mathbf w)=0$, so the firm labelled $\pi(1)$ can be successfully mobilized.

Now suppose that the first $(i-1)$ firms have already been mobilized. At the $i$-th step, the total accumulated height of the skyscraper built by the firms that have been mobilized is
$$
\sum_{j<i}c_{\pi(j)}.
$$
Moreover, the number of eligible elements for $\pi(i)$ that appear before $a_{\pi(i)}^{(-m_{\pi(i)})}$ in $\mathbf{w}$ is
$$
\phi_{\pi(i)}(\mathbf w)\leq \left|\left(\bigcup_{\substack{j<i\\\pi(j)<\pi(i)}}A_{\pi(j)}'\right)\cup \left(\bigcup_{\substack{j<i\\\pi(j)>\pi(i)}}A_{\pi(j)}\right)\right|=\sum_{j<i}c_{\pi(j)}.
$$
Hence, the firm labelled $\pi(i)$ can be successfully mobilized. By induction, the procedure yields a building sequence, proving that $\phi_{\mathbf c}$ is well defined.
\end{proof}

Before we state our next lemma, we assign a representative from the $\mathbf{c}$-Shi words to each region of the asymmetric $\mathbf{c}$-Shi arrangement.

\begin{definition}
We choose the \textit{representative} of each $\mathbf{c}$-Shi equivalence class to be the word in which, whenever $j>i$ and $a_i^{(m_i-\mathds{1}_I(i))}$ can be swapped with $a_j^{(-m_j)}$, the two alphabets must occur in the order $a_j^{(-m_j)}a_i^{(m_i-\mathds{1}_I(i))}$.
\end{definition}

\begin{example}
The region marked $B$ in Figure \ref{fig:shi_regions} corresponds to a $\mathbf{c}$-Shi equivalence class containing four Catalan $\mathbf{c}$-ordered words under the Shi move. Among these, $\mathbf{w}_4$ is the representative.
\end{example}

\begin{lemma}\label{lem:phi_m_inject}
Let $\mathbf{w}$ and $\mathbf{w'}$ be two Catalan $\mathbf{c}$-ordered word. Then $\mathbf{w}$ and $\mathbf{w}'$ are $\mathbf{c}$-Shi equivalent if and only if $\phi_{\mathbf{c}}(\mathbf{w})=\phi_{\mathbf{c}}(\mathbf{w}')$.
\end{lemma}

\begin{proof}
Suppose first that $\mathbf{w}$ and $\mathbf{w}'$ are $\mathbf{c}$-Shi equivalent. Since $\mathbf{c}$-Shi equivalence is generated by single Shi moves, it suffices to show that $\phi_{\mathbf c}$ is invariant under each such move. Thus, suppose that $\mathbf{w}'$ is obtained from $\mathbf{w}$ by a single Shi move. We show that $\phi_{\mathbf c}(\mathbf w)=\phi_{\mathbf c}(\mathbf w')$.
The general case then follows by applying this argument successively to each Shi move. Suppose that a single Shi move swaps the alphabets $a_{j_1}^{(m_{j_1}-\mathds{1}_I(j_1))}$ and $a_{j_2}^{(-m_{j_2})}$, where $j_1<j_2$. Observe that this swap does not affect the contribution to $\phi_r$ for any $r\in [n]\setminus \{j_1,j_2\}$. Therefore, 
$$
\phi_r(\mathbf{w})=\phi_r(\mathbf{w}')
$$
for all $r\in [n]\setminus\{j_1,j_2\}$. 

It remains to consider the cases $r=j_1$ and $r=j_2$. Since $a_{j_2}^{(-m_{j_2})}$ does not contribute to either $\phi_{j_1}(\mathbf{w})$ or $\phi_{j_1}(\mathbf{w}')$, and $a_{j_1}^{(-m_{j_1})}$ does not contribute to either $\phi_{j_2}(\mathbf{w})$ or $\phi_{j_2}(\mathbf{w}')$, as $a_{j_2}^{(-m_{j_2})}$ is not an eligible alphabet for the index $j_1$, and $a_{j_1}^{(-m_{j_1})}$ is not an eligible alphabet for the index $j_2$, we again obtain
$$
\phi_r(\mathbf{w})=\phi_r(\mathbf{w}')
$$
for $r\in\{j_1,j_2\}$. Hence, 
$$
\phi_{\mathbf{c}}(\mathbf{w})=\phi_{\mathbf{c}}(\mathbf{w}'),
$$ 
proving the converse implication.

Conversely, suppose that $\phi_{\mathbf c}(\mathbf w)=\phi_{\mathbf c}(\mathbf w')$. We have already proved that if $\mathbf{w}'$ is obtained from $\mathbf{w}$ by a single Shi move, then $\phi_{\mathbf c}(\mathbf w)=\phi_{\mathbf c}(\mathbf w')$. Therefore we may assume that $\mathbf{w}$ and $\mathbf{w}'$ are representative of their respective region. If possible let $\mathbf{w}$ and $\mathbf{w}'$ are not $\mathbf{c}$-Shi equivalent. Let $j$ be the first position from the left where $\mathbf{w}$ and $\mathbf{w}'$ differ. Let the elements in the $j$-th positions of $\mathbf{w}$ and $\mathbf{w}'$ be $a_{i_1}^{(l_1)}$ and $a_{i_2}^{(l_2)}$, respectively. Since $j$ is the first position from the left where the equality does not hold, all positions before $j$ are identical in $\mathbf{w}$ and $\mathbf{w}'$.

We claim that $l_1=-m_{i_1}$ or $l_2=-m_{i_2}$. Suppose, to the contrary, that $l_1\neq-m_{i_1}$ and $l_2\neq-m_{i_2}$. By the first condition in the Definition \ref{defn:gen_cat_m_ord_seqn}, $a_{i_1}^{(l_1-1)}$ appears before $j$-th position of $\mathbf{w}$. Similarly, $a_{i_2}^{(l_2-1)}$ occurs before position $j$ in $\mathbf{w}'$. Since all positions before $j$ are identical in $\mathbf{w}$ and $\mathbf{w}'$, the relative order of $a_{i_1}^{(l_1-1)}$ and $a_{i_2}^{(l_2-1)}$ is the same in $\mathbf{w}$ and $\mathbf{w}'$. Hence, by the second condition in the Definition \ref{defn:gen_cat_m_ord_seqn}, the relative order of $a_{i_1}^{(l_1)}$ and $a_{i_2}^{(l_2)}$ is also the same in $\mathbf{w}$ and $\mathbf{w}'$, respectively. This contradicts the fact that the $j$-th letters of $\mathbf{w}$ and $\mathbf{w}'$ are different. Therefore, $l_1=-m_{i_1}$ or $l_2=-m_{i_2}$.

Without loss of generality, suppose that $l_1=-m_{i_1}$, i.e. the $j$-th letter of $\mathbf{w}$ is $a_{i_1}^{(-m_{i_1})}$. Since $\mathbf{w}$ and $\mathbf{w}'$ agree in all positions before $j$, and the $j$-th letter of $\mathbf{w}'$ is different from $a_{i_1}^{(-m_{i_1})}$, the letter $a_{i_1}^{(-m_{i_1})}$ appears later in $\mathbf{w}'$.

\noindent\textbf{Case 1:} Suppose that $i_1<i_2$. If $l_2\neq-m_{i_2}$, then $a_{i_2}^{(l_2)}\in A_{i_1}$ is an eligible alphabet for $i_1$. Since $a_{i_2}^{(l_2)}$ occurs in the $j$-th position of $\mathbf{w}'$ and after $a_{i_1}^{(-m_{i_1})}$ in $\mathbf{w}$, this alphabet is counted by $\phi_{i_1}(\mathbf{w}')$ but not by $\phi_{i_1}(\mathbf{w})$. Therefore, by the definition of $\phi_{i_1}$,
\[
\phi_{i_1}(\mathbf{w})<\phi_{i_1}(\mathbf{w}'),
\]
which contradicts $\phi_{\mathbf c}(\mathbf{w})=\phi_{\mathbf c}(\mathbf{w}')$.

If $l_2=-m_{i_2}$, then $a_{i_1}^{(-m_{i_1})}$ is an eligible alphabet for $i_2$. Since $a_{i_1}^{(-m_{i_1})}$ occurs in the $j$-th position of $\mathbf{w}$ and after $a_{i_2}^{(-m_{i_2})}$ in $\mathbf{w}'$, this alphabet is counted by $\phi_{i_2}(\mathbf{w})$ but not by $\phi_{i_2}(\mathbf{w}')$. Therefore, by the definition of $\phi_{i_2}$,
\[
\phi_{i_2}(\mathbf{w}')<\phi_{i_2}(\mathbf{w}),
\]
which contradicts $\phi_{\mathbf c}(\mathbf{w})=\phi_{\mathbf c}(\mathbf{w}')$.

\medskip
\noindent\textbf{Case 2:} Suppose that $i_1>i_2$. If $l_2\neq m_{i_2}-\mathds{1}_I(i_2)$, then $a_{i_2}^{(l_2)}\in A_{i_1}'$ is an eligible alphabet for $i_1$. Since $a_{i_2}^{(l_2)}$ occurs in the $j$-th position of $\mathbf{w}'$ and after $a_{i_1}^{(-m_{i_1})}$ in $\mathbf{w}$, this alphabet is counted by $\phi_{i_1}(\mathbf{w}')$ but not by $\phi_{i_1}(\mathbf{w})$. Therefore,
\[
\phi_{i_1}(\mathbf{w})<\phi_{i_1}(\mathbf{w}'),
\]
which contradicts $\phi_{\mathbf c}(\mathbf{w})=\phi_{\mathbf c}(\mathbf{w}')$.

If $l_2=m_{i_2}-\mathds{1}_I(i_2)$, then, since $\mathbf{w}'$ is a representative, the alphabet immediately following $a_{i_2}^{(l_2)}$ cannot be of the form $a_{r_1}^{(-m_{r_1})}$ with $r_1>i_2$, as this would allow a Shi move, contradicting the fact that $\mathbf{w}'$ is a representative. Moreover, any alphabet occurring between $a_{i_2}^{(l_2)}$ and $a_{i_1}^{(-m_{i_1})}$ in $\mathbf{w}'$ cannot be eligible for $i_1$; otherwise, we would have
$$
\phi_{i_1}(\mathbf{w})<\phi_{i_1}(\mathbf{w}'),
$$
contradicting $\phi_{\mathbf c}(\mathbf{w}) =\phi_{\mathbf c}(\mathbf{w}')$.

Therefore, the alphabet immediately following $a_{i_2}^{(l_2)}$ must be of the form $a_s^{(m_s-\mathds{1}_I(s))}$ for some $s<i_1$. Applying the same argument to this alphabet, the next alphabet must again be of the same form, with an index smaller than $i_1$. Repeating this process, we either reach an exhausted index, which gives a contradiction, or eventually reach an alphabet immediately preceding $a_{i_1}^{(-m_{i_1})}$ of the form $a_s^{(m_s-\mathds{1}_I(s))}$ with $s<i_1$. In the latter case, this alphabet could be swapped with $a_{i_1}^{(-m_{i_1})}$, contradicting the fact that $\mathbf{w}'$ is a representative.

\medskip

Thus both cases are impossible, and hence $\mathbf{w}$ and $\mathbf{w}'$ are $\mathbf{c}$-Shi equivalent.
\end{proof}

By Theorem \ref{count_building-func} and Note \ref{same size R, S} we have the number of Shi $\mathbf{c}$-ordered words is equal to the number of $\mathbf{c}$-building function, i.e., 
$$|\mathcal{W}_{S}(\mathbf{c})|=|\mathcal{B}(\mathbf{c})|=( 1 + C )^{n-1}.$$ 
Therefore, to prove that the map $\tilde\phi_{\mathbf{c}}:\mathcal{W}_{S}(\mathbf{c})\rightarrow\mathcal{B}(\mathbf{c})$ is a bijection, it suffices to show that $\tilde\phi_{\mathbf{c}}$ is injective.

\begin{theorem}\label{thm:word_and_parking}
    The map $\tilde{\phi}_{\mathbf{c}}$ is a bijection between $\mathcal{W}_{S}(\mathbf{c})$ and $\mathcal{B}(\mathbf{c})$.
\end{theorem}

\begin{proof}
By Lemmas \ref{lem:ph1_1_well_defined} and \ref{lem:phi_m_inject}, the map $\tilde{\phi}_{\mathbf{c}}$ is well defined and injective, respectively. Hence, it is bijective.
\end{proof}

\noindent We now define the inverse of the map $\tilde{\phi}_{\mathbf{c}}$. 
\medskip

\noindent \textbf{The map $\left(\tilde{\phi}_{\mathbf{c}}\right)^{-1}: \mathcal{B}({\mathbf{c}}) \rightarrow \mathcal{W}_{S}{(\mathbf{c})}$:}  Let $\mathbf{i}=(i_1,\dots,i_n)$ be a given tuple.

\begin{enumerate}
    \item Construct $\pi\in S_n$ such that, for $j_1<j_2$, either $i_{\pi(j_1)}<i_{\pi(j_2)}$, or $i_{\pi(j_1)}=i_{\pi(j_2)}$ and $\pi(j_1)<\pi(j_2)$. Observe that $\pi$ is uniquely determined by the tuple $\mathbf{i}$. 

    \item We now recursively construct the ordering of the alphabets. Begin with: 
    $$
    a_{\pi(1)}^{(-m_{\pi(1)})} a_{\pi(1)}^{(-m_{\pi(1)}+1)}\dots a_{\pi(1)}^{(m_{\pi(1)}-\mathds{1}_I(\pi(1)))}.
    $$

    \item Suppose that the alphabets corresponding to $\pi(1),\pi(2),\dots,\pi(j-1)$ have already been placed. Place the alphabet $a_{\pi(j)}^{(-m_{\pi(j)})}$ immediately after the $i_{\pi(j)}$-th eligible alphabet for $\pi(j)$. Then place the remaining alphabets $a_{\pi(j)}^{(-m_{\pi(j)}+1)},\dots,a_{\pi(j)}^{(m_{\pi(j)}-\mathds{1}_I(\pi(1)))}$ so that it satisfies Definition \ref{defn:gen_cat_m_ord_seqn}.
\end{enumerate}

Let us illustrate the construction of the inverse map with an example.

\begin{example}
    For the tuple $\mathbf{i}=(0,2,1)$, with $\mathbf{c}=(4,2,1)$, the algorithm proceeds as follows.

\begin{itemize}
    \item Since $i_1=0$ and $i_3=1<i_2=4$, we obtain the permutation $\pi=132$.

    \item Place the alphabets corresponding to $\pi(1)=1$:
    $$
    a_1^{(-2)}a_1^{(-1)}a_1^{(0)}a_1^{(1)}a_1^{(2)}.
    $$

    \item Now we place the alphabets corresponding to $\pi(2)=3$. Since $i_3=1$, place $a_3^{(-1)}$ immediately after the $1$-st eligible alphabet for $3$. Then place
    $a_3^{(0)},a_3^{(1)}$
    so that Definition~\ref{defn:gen_cat_m_ord_seqn} is satisfied. This gives
    $$
    a_1^{(-2)}a_3^{(-1)}a_1^{(-1)}a_3^{(0)}
    a_1^{(0)}a_1^{(1)}a_1^{(2)}.
    $$

    \item Finally we place the alphabets corresponding to $\pi(3)=2$. Since $i_2=2$, place $a_2^{(-1)}$ immediately after the two eligible alphabets for $2$, namely after $a_1^{(-1)}$. Then place $a_2^{(0)},a_2^{(1)}$ so that Definition~\ref{defn:gen_cat_m_ord_seqn} is satisfied. The resulting word is
   $$a_1^{(-2)}a_3^{(-1)}a_1^{(-1)}a_2^{(-1)}a_3^{(0)}a_1^{(0)}a_2^{(0)}a_1^{(1)}a_2^{(1)}a_1^{(2)},$$
\end{itemize}
\end{example}

\noindent Here we give a rough sketch of the proof that $\left(\tilde{\phi}_{\mathbf{c}}\right)^{-1}$ is indeed the inverse of $\tilde{\phi}_{\mathbf{c}}$.

% \blue{In the algorithm, our aim is to place the alphabet corresponding to $i$ first, where the $i$-th firm has the smallest initiation threshold, namely $0$. We then apply the algorithm to the remaining firms in increasing order of their initiation thresholds. If several firms, say $i_1<i_2<\cdots<i_l$, have the same initiation threshold, we process them in the order $i_1,i_2,\ldots,i_l$. This tie-breaking rule is necessary to ensure that the algorithm is giving us the inverse map. Therefore, it is important to consider the permutation $\pi$ with this ordering convention.}
\begin{proof}[Sketch of the proof that $\tilde{\phi}_{\mathbf{c}}^{-1}$ is inverse of $\tilde{\phi}_{\mathbf{c}}$]
The construction proceeds recursively in the order $\pi(1),\ldots,\pi(n)$. At the $j$-th step, we insert the alphabet corresponding to $\pi(j)$ at the position prescribed by $i_{\pi(j)}$. This can be done because the number of alphabets in the current subword is $\sum_{r=1}^{j-1}c_{\pi(r)}$, and by the property of the $\mathbf{c}$-building function,
$$
i_{\pi(j)}\leq \sum_{r=1}^{j-1}c_{\pi(r)}.
$$
We then place its remaining alphabets so that Definition~\ref{defn:gen_cat_m_ord_seqn} is satisfied. This proves that the inverse map is well defined.

When $i_{\pi(j-1)}\neq i_{\pi(j)}$, all the new alphabets are placed after the alphabets 
$$
a_{\pi(1)}^{(-m_{\pi(1)})},\ldots, a_{\pi(j-1)}^{(-m_{\pi(j-1)})},
$$
and therefore the previously determined values $\phi_{\pi(1)}(\mathbf w),\ldots, \phi_{\pi(j-1)}(\mathbf w)$ remain unchanged.

When $i_{\pi(j-1)}=i_{\pi(j)}$, then $a_{\pi(j)}^{(-m_{\pi(j)})}$ is placed before
$a_{\pi(j-1)}^{(-m_{\pi(j-1)})}$ and after
$$
a_{\pi(1)}^{(-m_{\pi(1)})},\ldots,
a_{\pi(j-2)}^{(-m_{\pi(j-2)})}.
$$
The remaining alphabets corresponding to $\pi(j)$ are then placed after
$$
a_{\pi(1)}^{(-m_{\pi(1)})},\ldots, a_{\pi(j-1)}^{(-m_{\pi(j-1)})}.
$$
The previously determined values $\phi_{\pi(1)}(\mathbf w),\allowbreak\ldots,\phi_{\pi(j-1)}(\mathbf w)$ remain unchanged, since $a_{\pi(j)}^{(-m_{\pi(j)})}$ is not an eligible alphabet for $i_{\pi(j-1)}$.

Thus, by induction, after all $n$ steps we obtain a word $\mathbf w$ satisfying
$$
\phi_j(\mathbf w)=i_j\qquad\text{for all }j\in[n].
$$
Hence,
$$
\phi_{\mathbf c}(\mathbf w)=\mathbf i,
$$
and, since $\tilde{\phi}_{\mathbf c}$ is bijective, the constructed map
is precisely $(\tilde{\phi}_{\mathbf c})^{-1}$.

We now establish the second main result of this paper, that is, an explicit bijection between the set of regions of the asymmetric $\mathbf{c}$-Shi arrangement, $\mathcal{R}_{S}(\mathbf{c})$, and the set of $\mathbf{c}$-building functions, $\mathcal{B}(\mathbf{c})$.
\end{proof}
% \begin{theorem}\label{main_shi}
% The map $\tilde{\phi}_{\mathbf{c}}\circ\tilde{\varphi}_{\mathbf{c}}$ is a bijection between $\mathcal{R}_{S}(\mathbf{c})$ and $ \mathcal{B}(\mathbf{c})$.
% \end{theorem}

\begin{proof}[\textbf{Proof of Theorem \ref{main_shi0}}]
   By Theorems \ref{thm:shi_region_and_words} and \ref{thm:word_and_parking}, both $\tilde{\varphi}_{\mathbf{c}}$ and $\tilde{\phi}_{\mathbf{c}}$ are bijections. Therefore, their composition $\tilde{\phi}_{\mathbf{c}}\circ\tilde{\varphi}_{\mathbf{c}}$ is also a bijection between $\mathcal{R}_{S}(\mathbf{c})$ and $ \mathcal{B}(\mathbf{c})$. We provide a commutative diagram for more clarity.
\begin{figure}[H]
\begin{center}
    \begin{tikzpicture}
    [>=Stealth,node distance=2.5cm and 2.5cm]

\node (R) {$\mathcal{R}_{S}(\mathbf{c})$};
\node[right=of R] (A) {$\mathcal{W}_{S}(\mathbf{c})$};
\node[below=of A] (D) {$\mathcal{B}(\mathbf{c})$};

% Horizontal arrow
\draw[->] (R) -- node[above] {$\tilde{\varphi}_{\mathbf{c}}$} (A);

% Vertical arrow
\draw[->] (A) -- node[right] {$\tilde{\phi}_{\mathbf{c}}$} (D);

% Diagonal arrow
%\draw[->] (R) -- (D);
\draw[->] (R) -- node[below,left] {$\tilde{\phi}_{\mathbf{c}}\circ\tilde{\varphi}_{\mathbf{c}}$} (D);

% Composition label
%\node[left=1.2cm of D] {$\psi_m\circ\phi_m$};

\end{tikzpicture}
\end{center}
\caption{Commutative diagram of bijection}
\end{figure}
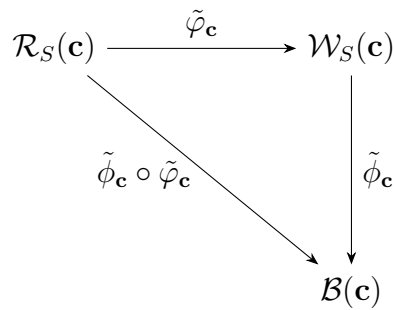
\end{proof}

\section{Conclusion and Open Issues}
In this paper, we developed a combinatorial framework for studying the regions of the asymmetric $\mathbf c$-Catalan and Shi arrangement. We introduced two new combinatorial objects, namely $\mathbf{c}$-labelled Dyck paths and $\mathbf{c}$-building functions, and established explicit bijections between these objects and the regions of the corresponding arrangements. Our results generalize several classical bijections, including Stanley's bijection between parking functions and Shi arrangements, using entirely different combinatorial constructions. Although these bijections provide a complete combinatorial interpretation of the region counts for the type A families considered here, they also suggest several interesting directions for future research.

% Some of the questions arising from this work are listed below.

\begin{question}
Let $\mathbf{m}=(m_{i,j})_{i,j\in[n]}\in\mathbb{Z}_{>0}^{n\times n}$ be a symmetric array satisfying the triangle inequality
$m_{i,j}\leq m_{i,k}+m_{k,j}$ for all $i,j,k\in[n]$.

\begin{definition}\label{def-array-cat-shi}
The following two hyperplane arrangements in $\mathbb{R}^n$ are defined by:
\begin{itemize}
    \item The \textit{$\mathbf{m}$-array Catalan arrangement}:
    $$\mathcal{A}^{\mathbf{m}}_{\mathrm{Cat}}
    =
    \{x_i-x_j=s\mid s\in[-m_{i,j},m_{i,j}],\ 1\leq i<j\leq n\}.$$

    \item The \textit{$\mathbf{m}$-array Shi arrangement}:
    $$\mathcal{A}^{\mathbf{m}}_{\mathrm{Shi}}
    =
    \{x_i-x_j=s\mid s\in[-m_{i,j}+1,m_{i,j}],\ 1\leq i<j\leq n\}.$$
\end{itemize}
\end{definition}

For $\sigma\in S_n$ with $\sigma(1)=1$, define
$$
S_{\sigma}^{\mathbf{m}}
=
\sum_{i=1}^{n-1}m_{\sigma(i),\sigma(i+1)}
+
m_{\sigma(1),\sigma(n)},
$$
and let $p_\sigma$ denote the number of descents of $\sigma$.

\begin{theorem}\label{thm:gen_cat_shi_char_func}
The characteristic polynomials of the above arrangements are given by:
\begin{itemize}
    \item For the $\mathbf{m}$-array Catalan arrangement,
    $$
    \chi(\mathcal{A}^{\mathbf{m}}_{\mathrm{Cat}},q)
    =
    \frac{q}{(n-1)!}
    \sum_{\substack{\sigma\in S_n\\ \sigma(1)=1}}
    \prod_{i=1}^{n-1}
    \bigl(q-S_{\sigma}^{\mathbf{m}}-i\bigr).
    $$

    \item For the $\mathbf{m}$-array Shi arrangement,
    $$
    \chi(\mathcal{A}^{\mathbf{m}}_{\mathrm{Shi}},q)
    =
    q
    \sum_{\substack{\sigma\in S_n\\ \sigma(1)=1}}
    \binom{q-S_{\sigma}^{\mathbf{m}}+p_{\sigma}}{n-1}.
    $$
\end{itemize}
\end{theorem}

Find combinatorial interpretations for the regions of the
$\mathbf{m}$-array Catalan and $\mathbf{m}$-array Shi arrangements.
\end{question}

\begin{question}
In this paper, we have focused on hyperplane arrangements of type $A$. It is natural to ask whether analogous results hold for arrangements of types $B$, $C$, and $D$. In particular, it would be interesting to determine their characteristic polynomials and identify explicit combinatorial models for their regions. A natural direction for future work is therefore to extend the bijective approach developed here to asymmetric versions of the type $B$, $C$, and $D$ arrangements.

% Deformations of these arrangements can be defined systematically using standard conventions from the literature. For instance, using a parameter $\mathbf{c}=(c_1,c_2,\ldots,c_n)\in\mathbb{Z}_{>0}^n$, define
% $$
% m_i=\begin{cases}
% \dfrac{c_i+1}{2}, & \text{if $c_i$ is odd},\\[6pt]
% \dfrac{c_i}{2}, & \text{if $c_i$ is even},
% \end{cases}
% $$
% and let $I=\{i\in[n]\mid c_i\text{ is odd}\}$. the asymmetric $\mathbf{c}$-Catalan arrangement  for the classical signed reflection groups in $\mathbb{R}^n$ are defined as follows:
% \begin{itemize}
%     \item The asymmetric $\mathbf{c}$-Catalan arrangement of  type B in $\mathbb{R}^n$ is given by
%     \begin{align*}
% &\{x_i \in [-m_i-m_j+\mathds{1}_{I}(i),\,m_i+m_j]
%    \mid i\in[n]\}\\
% &\qquad{}\cup
% \{x_i+x_j \in [-m_i-m_j+\mathds{1}_{I}(i),\,
% m_i+m_j-\mathds{1}_{I}(j)]\}
% \cup \mathcal{A}^{\mathbf{c}}_{\mathrm{Cat}}.
% \end{align*}

%     \item The asymmetric $\mathbf{c}$-Catalan arrangement of  type C in $\mathbb{R}^n$ is given by
%     \begin{align*}
% &\{2x_i \in [-m_i-m_j+\mathds{1}_{I}(i),\,m_i+m_j]
%    \mid i\in[n]\}\\
% &\qquad{}\cup
% \{x_i+x_j \in [-m_i-m_j+\mathds{1}_{I}(i),\,
% m_i+m_j-\mathds{1}_{I}(j)]\}
% \cup \mathcal{A}^{\mathbf{c}}_{\mathrm{Cat}}.
% \end{align*}

%     \item The asymmetric $\mathbf{c}$- Catalan arrangement of  type D in $\mathbb{R}^n$ is given by
%     $$ \{x_i +x_j \in [ -m_i-m_j+\mathds{1}_{I}(i),\, m_i+m_j-\mathds{1}_{I}(j)] \}\cup \mathcal{A}^{\mathbf{c}}_{\mathrm{Cat}}.$$
% \end{itemize}
% Analogous definitions can be given for the   asymmetric $\mathbf{c}$-Shi arrangement. 
\end{question}

\section{Acknowledgements}

This project was initiated during the NCM Workshop on Topics in Geometric Combinatorics, held at CMI, Chennai, in July 2025. The authors are grateful to the organizers for providing a productive and inspiring research environment. They would especially like to thank Dr. Priyavrat Deshpande, one of the organizers of the workshop, for his valuable discussions and insightful suggestions. The authors also sincerely thank Dr. Hiranya Kishor Dey for the initial discussions that contributed to this work. We are thankful to Dr. Umesh Shankar for  observing that the idea used for the hyperplane arrangement in an earlier version of this paper extends to the more general class of hyperplane arrangements considered in the present paper, and for bringing this observation to our attention.

We acknowledge the use of generative AI tools strictly to improve readability and grammatical flow. All conceptual ideas, computations, and mathematical proofs were developed solely by the authors.

\bibliographystyle{plain}
\bibliography{shu2}

@misc{margoliash2010matrix,
  author    = {Margoliash, Jonathan},
  title     = {The {{M}}atrix--{{T}}ree {{T}}heorem for directed graphs},
  year      = {2010},
  note      = {Expository note, Stanford University},
  url       = {https://math.stanford.edu/~jmarg/matrix_tree.pdf}
}

@book{crapo1970foundations,
  author    = {Crapo, Henry H. and Rota, Gian-Carlo},
  title     = {On the Foundations of Combinatorial Theory: Combinatorial Geometries},
  publisher = {The MIT Press},
  address   = {Cambridge, MA},
  year      = {1970}
}

@article{athanasiadis1996characteristic,
  author    = {Athanasiadis, Christos A.},
  title     = {Characteristic polynomials of subspace arrangements and finite fields},
  journal   = {Advances in Mathematics},
  volume    = {122},
  number    = {2},
  pages     = {193--233},
  year      = {1996},
  doi       = {10.1006/aima.1996.0059}
}

@book{zaslavsky1975facing,
  author    = {Zaslavsky, Thomas},
  title     = {Facing Up to Arrangements: Face-Count Formulas for Partitions of Space by Hyperplanes},
  series    = {Memoirs of the American Mathematical Society},
  volume    = {154},
  publisher = {American Mathematical Society},
  address   = {Providence, RI},
  year      = {1975},
  doi       = {10.1090/memo/0154}
}

@incollection{stanley1998hyperplane,
  author    = {Stanley, Richard P.},
  title     = {Hyperplane arrangements, parking functions and tree inversions},
  booktitle = {Mathematical Essays in Honor of Gian-Carlo Rota},
  editor    = {Sagbi, Bruce E. and Yan, Catherine H.},
  series    = {Progress in Mathematics},
  volume    = {161},
  publisher = {Birkh{\"a}user},
  address   = {Boston, MA},
  pages     = {359--375},
  year      = {1998},
  doi       = {10.1007/978-1-4612-2022-0_18}
}

@article{bernardi2018deformations,
  author    = {Bernardi, Olivier},
  title     = {Deformations of the braid arrangement and trees},
  journal   = {Advances in Mathematics},
  volume    = {335},
  pages     = {466--518},
  year      = {2018},
  doi       = {10.1016/j.aim.2018.07.014}
}

@article{deshpande2025sketches,
  author    = {Deshpande, Priyavrat and Menon, Krishna},
  title     = {Sketches, moves and partitions: counting regions of {{C}}atalan deformations of reflection arrangements},
  journal   = {The Electronic Journal of Combinatorics},
  volume    = {32},
  number    = {1},
  pages     = {P1.46},
  year      = {2025},
  doi       = {10.37236/13352}
}

@book{orlik1992arrangements,
  author    = {Orlik, Peter and Terao, Hiroaki},
  title     = {Arrangements of Hyperplanes},
  series    = {Grundlehren der mathematischen Wissenschaften},
  volume    = {300},
  publisher = {Springer-Verlag},
  address   = {Berlin, Heidelberg},
  year      = {1992},
  doi       = {10.1007/978-3-662-02772-1}
}

@book{humphreys1990reflection,
  author    = {Humphreys, James E.},
  title     = {Reflection Groups and {{C}}oxeter Groups},
  series    = {Cambridge Studies in Advanced Mathematics},
  volume    = {29},
  publisher = {Cambridge University Press},
  address   = {Cambridge},
  year      = {1990},
  doi       = {10.1017/CBO9780511623646}
}

@article{stanley1996hyperplane,
  author    = {Stanley, Richard P.},
  title     = {Hyperplane arrangements, interval orders, and trees},
  journal   = {Proceedings of the National Academy of Sciences of the United States of America},
  volume    = {93},
  number    = {6},
  pages     = {2620--2625},
  year      = {1996},
  doi       = {10.1073/pnas.93.6.2620}
}

@book{shi1986kazhdan,
  author    = {Shi, Jian-Yi},
  title     = {The {{K}}azhdan--{{L}}usztig Cells in Certain Affine {{W}}eyl Groups},
  series    = {Lecture Notes in Mathematics},
  volume    = {1179},
  publisher = {Springer-Verlag},
  address   = {Berlin, Heidelberg},
  year      = {1986},
  doi       = {10.1007/BFb0074968}
}

@article{Alex-Boris,
  author    = {Postnikov, Alexander and Shapiro, Boris},
  title     = {Trees, parking functions, syzygies, and deformations of monomial ideals},
  journal   = {Transactions of the American Mathematical Society},
  volume    = {356},
  number    = {8},
  pages     = {3109--3142},
  year      = {2004},
  doi       = {10.1090/S0002-9947-04-03541-6}
}

@article{postnikov2000deformations,
  author    = {Postnikov, Alexander and Stanley, Richard P.},
  title     = {Deformations of {{C}}oxeter hyperplane arrangements},
  journal   = {Journal of Combinatorial Theory, Series A},
  volume    = {91},
  number    = {1--2},
  pages     = {544--597},
  year      = {2000},
  doi       = {10.1006/jcta.2000.3106}
}

@article{stanley1973acyclic,
  author    = {Stanley, Richard P.},
  title     = {Acyclic orientations of graphs},
  journal   = {Discrete Mathematics},
  volume    = {5},
  number    = {2},
  pages     = {171--178},
  year      = {1973},
  doi       = {10.1016/0012-365X(73)90108-8}
}

@article{bernardi2026bijectivity,
  author    = {Bernardi, Olivier and Goregaokar, Neha},
  title     = {Bijectivity of a generalized {{P}}ak--{{S}}tanley labeling},
  journal   = {arXiv preprint arXiv:2603.24886},
  year      = {2026},
  eprint    = {2603.24886},
  archivePrefix = {arXiv},
  primaryClass = {math.CO}
}

@incollection{williams2022oberwolfach,
  author    = {Williams, Nathan},
  title     = {Problem Session},
  booktitle = {Enumerative Combinatorics},
  series    = {Oberwolfach Reports},
  volume    = {19},
  number    = {4},
  publisher = {European Mathematical Society},
  pages     = {3101--3104},
  year      = {2022},
  doi       = {10.4171/OWR/2022/53}
}

@inproceedings{bernardi2025deformations,
  author    = {Bernardi, Olivier and Douvropoulos, Theo},
  title     = {Deformations of restricted reflection arrangements},
  booktitle = {Proceedings of the 37th International Conference on Formal Power Series and Algebraic Combinatorics (FPSAC 2025)},
  series    = {S{\'e}minaire Lotharingien de Combinatoire},
  volume    = {93B},
  year      = {2025}
}

@article{duarte2021pak,
  title={Pak-Stanley labeling of the m-Catalan hyperplane arrangement},
  author={Duarte, Rui and de Oliveira, Ant{\'o}nio Guedes},
  journal={Advances in Mathematics},
  publisher={Elsevier},
  volume={387},
  pages={107827},
  year={2021}
  
}

@article{fu2021bijections,
  title={Bijections on r-Shi and r-Catalan arrangements},
  author={Fu, Houshan and Wang, Suijie and Zhu, Weijin},
  journal={Advances in Applied Mathematics},
  publisher={Elsevier},
  volume={129},
  pages={102207},
  year={2021}
  
}

\end{document}